\documentclass[11pt]{article}
\usepackage{epigamath}

\usepackage[notext]{kpfonts}
\usepackage{baskervald}

\setpapertype{A4}

\usepackage[english]{babel}

\usepackage[dvips]{graphicx}     

\removelength{1cm}

\title{Higher rank sheaves on threefolds and functional equations}
\titlemark{Higher rank sheaves and functional equations}
\author{Amin Gholampour and Martijn Kool}
\authoraddresses{
\authordata{Amin Gholampour}{\firstname{Amin} \lastname{Gholampour}\\
\institution{Department of Mathematics, University of Maryland, College Park, MD 20742-4015, USA}\\
\email{amingh@umd.edu}} \\
\authordata{Martijn Kool}{\firstname{Martijn} \lastname{Kool}\\
\institution{Mathematical Institute, Utrecht University,
3584 CD Utrecht, The Netherlands}\\
\email{m.kool1@uu.nl}}
}
\authormark{A. Gholampour and M. Kool}
\date{\vspace{-5ex}} 
\journal{\'Epijournal de G\'eom\'etrie Alg\'ebrique} 
\acceptation{Received by the Editors on March 15, 2018, and in final form on April 26, 2019. \\ Accepted on October 1, 2019.}

\newcommand{\articlereference}{Volume 3 (2019), Article Nr. 17}

\usepackage[all]{xy}

\allowdisplaybreaks

\newdimen\origiwspc
\origiwspc=\fontdimen2\font

\newcommand{\extendspace}[1]{\fontdimen2\font=1.2ex #1 \fontdimen2\font=\origiwspc}

\usepackage{float}

\makeatletter

\renewcommand{\p@equation}{\arabic{equation}\expandafter\@gobble}
\makeatother

\newtheorem{theorem}{Theorem}[section]

\newtheorem{proposition}[theorem]{Proposition}
\newtheorem{lemma}[theorem]{Lemma}

\newtheorem{remark}[theorem]{Remark}

\DeclareFontFamily{OT1}{rsfs}{}
\DeclareFontShape{OT1}{rsfs}{n}{it}{<-> rsfs10}{}
\DeclareMathAlphabet{\curly}{OT1}{rsfs}{n}{it}

\renewcommand\O{\mathcal O}
\newcommand\PP{\mathbb P}

\newcommand\cE{\mathcal E}
\newcommand\F{\mathcal F}
\newcommand\cK{\mathcal K}
\newcommand\T{\mathcal T}

\newcommand\C{\mathbb C}
\newcommand\cC{\mathcal C}

\newcommand\sfM{\mathsf M}
\newcommand\Q{\mathbb Q}
\newcommand\cQ{\mathcal Q}

\newcommand\ccR{\mathcal R}
\newcommand\Z{\mathbb Z}

\newcommand\Quot{\mathrm{Quot}}
\newcommand\onto{\mathrm{onto}}
\newcommand\pure{\mathrm{pure}}

\newcommand\PT{\mathrm{PT}}

\newcommand\ch{\operatorname{ch}}

\newcommand\Hom{\operatorname{Hom}}
\renewcommand\hom{\curly H\!om}
\newcommand\Ext{\operatorname{Ext}}
\newcommand\ext{\curly Ext}

\newcommand\Proj{\operatorname{Proj}\,}
\newcommand\Spec{\operatorname{Spec}\,}
\newcommand\Hilb{\operatorname{Hilb}}

\newcommand\Sym{\operatorname{Sym}}

\newcommand\mdot{{\scriptscriptstyle\bullet}}

\newcommand\INTO{\ar@{^{(}->}[r]}

\DeclareRobustCommand{\SkipTocEntry}[4]{}

\newcommand{\qedhere}{}

\begin{document}


\maketitle



\begin{prelims}

\vspace{-0.55cm}

\def\abstractname{Abstract}
\abstract{We consider the moduli space of stable torsion free sheaves of any rank on a smooth projective threefold. The singularity set of a torsion free sheaf is the locus where the sheaf is not locally free. On a threefold it has dimension $\leq 1$. We consider the open subset of moduli space consisting of sheaves with empty or 0-dimensional singularity set. \\
For fixed Chern classes $c_1,c_2$ and summing over $c_3$, we show that the generating function of topological Euler characteristics of these open subsets equals a power of the MacMahon function times a Laurent polynomial. This Laurent polynomial is invariant under $q \leftrightarrow q^{-1}$ (upon replacing $c_1 \leftrightarrow -c_1$). For some choices of $c_1,c_2$ these open subsets equal the entire moduli space. \\
The proof involves wall-crossing from Quot schemes of a higher rank reflexive sheaf to a sublocus of the space of Pandharipande-Thomas pairs. We interpret this sublocus in terms of the singularities of the reflexive sheaf.}




\languagesection{Fran\c{c}ais}{%

\vspace{-0.05cm}
{\bf Titre. Faisceaux de rang sup\'erieur sur les solides et \'equations fonctionnelles} \commentskip {\bf R\'esum\'e.} Nous consid\'erons l'espace de modules des faisceaux stables et sans torsion de rang quelconque sur un solide projectif. L'ensemble singulier d'un faisceau sans torsion est le lieu o\`u ce faisceau n'est pas localement libre. Sur un solide, ce lieu est de dimension $\leq 1$. Nous consid\'erons l'ouvert de l'espace de modules constitu\'e par les faisceaux de lieu singulier vide ou de dimension nulle. \\
Pour des classes de Chern $c_1,\,c_2$ fix\'ees et en sommant sur $c_3$, nous montrons que la fonction g\'en\'eratrice des caract\'eristiques d'Euler topologiques de ces ouverts est \'egale au produit d'une puissance de la fonction de MacMahon par un polyn\^ome de Laurent. Ce dernier est invariant par $q \leftrightarrow q^{-1}$ (quitte \`a remplacer $c_1$ par $-c_1$). Pour certains choix de $c_1,c_2$, ces ouverts co\"{\i}ncident avec l'espace de modules tout entier. \\
La d\'emonstration utilise des travers\'ees de murs \`a partir des sch\'emas Quot d'un faisceau r\'eflexif de rang sup\'erieur vers un sous-lieu de l'espace des couples de Pandharipande-Thomas. Nous interpr\'etons ce sous-lieu en termes des singularit\'es du faisceau r\'eflexif.}

\end{prelims}


\newpage

\setcounter{tocdepth}{1} \tableofcontents

\section{Introduction} \label{intro} 

Let $X$ be a smooth projective threefold. We consider torsion free sheaves $\F$ of homological dimension $\leq 1$ on $X$, i.e.~torsion free sheaves which are locally free or can be resolved by a 2-term complex of locally free sheaves. They have the property that $\ext^{>1}(\F,\O_X) = 0$ and $\ext^1(\F,\O_X)$ has dimension $\leq 1$. Important examples are reflexive sheaves. They are precisely the torsion free sheaves $\F$ of homological dimension $\leq 1$ for which $\ext^1(\F,\O_X)$ is zero or 0-dimensional. We refer to \cite[Prop.~1.1.10]{HL} for details.

We first study Quot schemes $\Quot_X(\F,n)$ 
of length $n$ quotients $\F \twoheadrightarrow \cQ$. We start with some notation. The reciprocal of a polynomial $P(q)$ of degree $d$ is 
$$
P^*(q) := q^d P(q^{-1}).
$$ 
Moreover $P$ is called palindromic when $P^*(q) = P(q)$. We denote the MacMahon function by
$$
\sfM(q) := \prod_{n=1}^{\infty} \frac{1}{(1-q^n)^n}.
$$
Furthermore $e(\cdot)$ denotes topological Euler characteristic.

\begin{theorem} \label{main}
For any rank $r$ torsion free sheaf $\F$ of homological dimension $\leq 1$ on a smooth projective threefold $X$, we have
\begin{equation*} 
\sum_{n=0}^{\infty} e\big(\Quot_X(\F,n)\big) \, q^n = \sfM(q)^{r e(X)} \sum_{n=0}^{\infty} e\big(\Quot_X(\ext^1(\F,\O_X),n)\big) \, q^n.
\end{equation*}
\end{theorem} 

We expect that the generating function of Euler characteristics of Quot schemes of 0-dimensional quotients of $\ext^1(\F,\O_X)$ in Theorem \ref{main} is always a rational function. We prove this in the toric case in a sequel \cite{GK4}. In any case, when $\F = \ccR$ is reflexive it is a polynomial and satisfies a nice duality.
\begin{theorem} \label{cor}
Let $\ccR$ be a rank $r$ reflexive sheaf on a smooth projective threefold $X$. Taking $\F = \ccR$ in Theorem \ref{main}, the right-hand side is $\sfM(q)^{r e(X)}$ times a polynomial, which we denote by $P_{\ccR}(q)$. This polynomial satisfies
$P^*_{\ccR}(q) = P_{\ccR^*}(q)$. Moreover $P_{\ccR}(q)$ is palindromic when $r=2$.
\end{theorem}

For a polarization $H$ on $X$, denote by $M_{X}^{H}(r,c_1,c_2,c_3)$ the moduli space of $\mu_H$-stable rank $r$ torsion free sheaves on $X$ with Chern classes $c_1,c_2,c_3$. We recall that the slope (with respect to $H$) of a rank $r$ torsion free sheaf $\F$ on $X$ is defined by $\mu_H(\F) = (c_1(\F) \cdot H^2) / r$. See \cite[Sect.~1.2]{HL} for details. For an element $[\F]$ of this moduli space, there is a natural inclusion of $\F$ into its double dual $\F^{**}$, called its reflexive hull, and the quotient $\F^{**} / \F$ has dimension $\leq 1$. Let
\begin{equation} \label{opensubset}
M_{X}^{H}(r,c_1,c_2,c_3)^\circ \subset M_{X}^{H}(r,c_1,c_2,c_3)
\end{equation}
be the (possibly empty) open subset of isomorphism classes $[\F]$ for which $\F^{**} / \F$ is zero or 0-dimensional. We prove openness in Lemma \ref{Kollar}. 

The open subset \eqref{opensubset} has an alternative description as follows. The singularity set of a coherent sheaf $\F$ is defined as the locus where $\F$ is not locally free \cite{OSS}. When $\F$ is torsion free on $X$, the singularity set has dimension $\leq 1$. In Lemma \ref{Kollar} we show that $M_{X}^{H}(r,c_1,c_2,c_3)^\circ$ is also the locus of sheaves with empty or 0-dimensional singularity set.

\begin{theorem} \label{thm2}
For any smooth projective threefold $X$ with polarization $H$
$$
P_{r,c_1,c_2}(q) := \frac{\sum_{c_3} e\big(M_{X}^{H}(r,c_1,c_2,c_3)^\circ \big) \, q^{c_3}}{\sfM(q^{-2})^{r e(X)}} 
$$
is a Laurent polynomial in $q$ satisfying 
$$
P_{r,c_1,c_2}(q^{-1}) = P_{r,-c_1,c_2}(q).
$$ 
Moreover for $r=2$
$$
P_{2,c_1,c_2}(q^{-1}) = P_{2,c_1,c_2}(q).
$$ 
\end{theorem}

For fixed $c_1 \in H^2(X,\Z)$, let $c_2 \in H^4(X,\Z)$ be chosen such that $c_2 \cdot H$ is the \emph{minimal} value for which there exist rank $r$ $\mu_H$-stable torsion free sheaves on $X$ with Chern classes $c_1,c_2$ or $-c_1,c_2$. Such a minimal value exists by the Bogomolov inequality \cite[Thm.~7.3.1]{HL}. For such a choice of $c_2$ we have 
$$
M_{X}^{H}(r,\pm c_1,c_2,c_3)^\circ = M_{X}^{H}(r,\pm c_1,c_2,c_3)
$$ 
for all $c_3$ (e.g.~by \cite[Prop.~3.1]{GKY}). In this case Theorem \ref{thm2} gives a functional equation for the \emph{complete} generating function. Explicit examples of the Laurent polynomials appearing in Theorem \ref{thm2} for $X$ a smooth projective toric threefold can be found in \cite[Ex.~3.7--3.9]{GKY}. E.g.~for $X = \PP^3$, $r=2$, $c_1=H$, $c_2=H^2$, where $H$ denotes the hyperplane class, we have $$P_{2,c_1,c_2}(q) = 4q^{-1}+4q.$$ 

\subsection*{The rank 1 case}

Suppose $C \subset X$ is a Cohen-Macaulay curve on a smooth projective threefold. Then its ideal sheaf $I_C$ has homological dimension 1. By dualizing the ideal sheaf sequence, we obtain an isomorphism
$$\ext^1(I_C,\O_X) \cong \ext^2(\O_C,\O_X).$$
In turn, dualizing any short exact sequence
$$
0 \rightarrow F' \rightarrow \ext^2(\O_C,\O_X) \rightarrow Q' \rightarrow 0,
$$
where $Q'$ is 0-dimensional, produces a Pandharipande-Thomas pair \cite{PT1}
$$
0 \rightarrow \O_C \rightarrow F \rightarrow Q \rightarrow 0,
$$
where $F:=\ext^2(F',\O_X)$, $Q := \ext^3(Q',\O_X)$, and $$\ext^2(\ext^2(\O_C,\O_X),\O_X) \cong \O_C.$$ This uses \cite[Prop.~1.1.6, 1.1.10]{HL}. These arguments show that the Quot schemes on the RHS in Theorem \ref{main} are in bijective correspondence\footnote{Here we only described a set theoretic bijection. This can be made into a morphism of schemes which is bijective on $\C$-valued points much like in Theorem \ref{key}.} to moduli spaces $\PT_{X}(C,n)$ of Pandharipande-Thomas pairs $(F,s)$ on $X$ with support curve $C$ and $\chi(F) = \chi(\O_C)+n$. Moreover, the Quot schemes on the LHS in Theorem \ref{main} are in bijective correspondence with moduli spaces $I_{X}(C,n)$ of ideal sheaves $I_Z \subset I_C$ such that $I_C/I_Z$ is 0-dimensional of length $n$. So in this case, Theorem \ref{main} reduces to the DT/PT correspondence on a fixed Cohen-Macaulay curve $C$ proved by Stoppa-Thomas \cite{ST}
\begin{equation*} 
\sum_{n=0}^{\infty} e\big(I_X(C,n)\big) \, q^n = \sfM(q)^{e(X)} \sum_{n=0}^{\infty} e\big(\PT_X(C,n)\big) \, q^n.
\end{equation*}
When $X$ is Calabi-Yau, a weighted Euler characteristic version of this formula was proved by A.~Ricolfi (for smooth $C$ \cite{Ric1, Ric2}) and G.~Oberdieck (for any CM curve $C$ \cite{Obe}).
Theorem \ref{main} can be seen as a generalization of Stoppa-Thomas's result to \emph{arbitrary torsion free sheaves of homological dimension 1}. 

Suppose $X$ is Calabi-Yau, $r=1$, and $c_1=0$. Then $M_{X}^{H}(1,0,c_2,c_3)$ are Hilbert schemes of subschemes of dimension $\leq 1$ of $X$ and 
$$
\frac{\sum_{c_3} e\big(M_{X}^{H}(1,0,c_2,c_3) \big) \, q^{c_3}}{\sfM(q^{-2})^{e(X)}} 
$$
is a rational function invariant under $q \leftrightarrow q^{-1}$. 
This was proved by Y.~Toda \cite{Tod1}. A Behrend function version of this statement was proved by T.~Bridgeland \cite{Bri2}. This established the famous rationality and functional equation of Pandharipande-Thomas theory \cite{PT1}. For general $X,r,c_1,c_2$ we expect
$$
\frac{\sum_{c_3} e\big(M_{X}^{H}(r,c_1,c_2,c_3) \big) \, q^{c_3}}{\sfM(q^{-2})^{re(X)}}
$$
is a rational function and again we prove it in the toric case in \cite{GK4}. Examples of this rational function were calculated in \cite{GKY} and they are certainly \emph{not} invariant under $q \rightarrow q^{-1}$ in general \cite[Ex.~3.8]{GKY}. 

\subsection*{The proof} Theorems \ref{cor}, \ref{thm2} are proved in Section \ref{section1}. A special case of Theorem \ref{main} was proved in \cite{GK2}, i.e.~for $\F = \ccR$ a rank 2 reflexive sheaf satisfying the following two additional conditions\footnote{J.~Rennemo pointed out that these two additional conditions can be dropped \cite{GK2}. After twisting by $-mH$, with $H$ very ample and $m \gg 0$, the desired cosection always exists. This argument only works for \emph{rank 2 reflexive sheaves}. Instead, in this paper we use that such a cosection always exists locally for \emph{torsion free sheaves of any rank} as discussed below.} 
\begin{itemize}
\item $H^1(\det \ccR) = H^2(\det \ccR) = 0$, 
\item there exists a cosection $\ccR \rightarrow \O_X$ such that the image is the ideal sheaf of a 1-dimensional scheme.
\end{itemize}
The proof of \cite{GK2} uses the Serre correspondence and a rank 2 version of a Hall algebra calculation by Stoppa-Thomas \cite{ST}. In this paper we first reduce to an \emph{affine version} of Theorem \ref{main} on $\Spec A$ (Section \ref{section1}) with
\begin{itemize}
\item $\Spec A$ a Zariski open subset of $X$, or
\item $\Spec A = \Spec \widehat{\O}_{X,P}$, where $\widehat{\O}_{X,P}$ is the completion of the stalk $\O_{X,P}$ at a closed point $P \in X$.
\end{itemize}
Let $M$ be the $A$-module corresponding to $\F |_{\Spec A}$. By a theorem of Bourbaki (see Theorem \ref{Bourbaki}), there exists an $A$-module homomorphism $M \rightarrow A$ with free kernel. This cosection $M \rightarrow A$ allows us to prove an affine version of Theorem \ref{main} in the rank 2 case following ideas of \cite{GK2}. The higher rank case builds on the rank 2 case by a new inductive construction. The proof of the affine version of Theorem \ref{main} occupies Sections \ref{section2} and \ref{Hall}. 

\begin{remark}
{\rm
It is an interesting question to what extent the methods of this proof work without reduction to the affine case $\Spec A$. Let $X$ be a smooth projective threefold and $\F$ a torsion free sheaf of homological dimension $1$ on $X$. Suppose there exists a cosection $\F \rightarrow \O_X$ with locally free kernel. Then it appears that most methods of the proof of this paper (in particular Theorem \ref{key}) work globally on $X$ as well. As mentioned above, when $\F = \ccR$ is reflexive of rank 2, such a cosection exists (possibly after twisting by a line bundle). In general, the authors only have the required cosection in the affine case (Theorem \ref{Bourbaki}), which leads to the current approach.}
\end{remark}

We expect that Theorem \ref{main} is related to Toda's recent work on the higher rank DT/PT correspondence on Calabi-Yau threefolds \cite{Tod2}, which involves J.~Lo's notion of higher rank Pandharipande-Thomas pairs \cite{Lo1}. The latter were also related to Quot schemes of $\ext^1(\ccR,\O_X)$, where $\ccR$ is a $\mu_H$-stable reflexive sheaf, in \cite[Rem.~4.4]{Lo1}. Recently, after this paper appeared, the relationships between these works have been clarified and further investigated by Lo \cite{Lo2} and S.~Beentjes and Ricolfi \cite{BR}.

Theorem \ref{main} has the nice property that it holds for any torsion free sheaf of homological dimension $\leq 1$ (not necessarily stable or reflexive) on any smooth projective threefold (not necessarily Calabi-Yau).

\bigskip

\noindent \textbf{Acknowledgements.} We thank J.~Lo, D.~Maulik, R.~Skjelnes, and R.~P.~Thomas for useful discussions. A.G.~was partially supported by NSF grant DMS-1406788. M.K.~was supported by Marie Sk{\l}odowska-Curie Project 656898.

\bigskip

\noindent \textbf{Notation.} All sheaves in this paper are coherent and all modules are finitely generated. The length of a 0-dimensional sheaf $\cQ$ is denoted by $\ell(\cQ)$.

For any sheaf $\cE$ with support of codimension $c$ on a smooth variety $X$  and any line bundle $L$ on $X$, we define
\begin{equation*} 
(\cE)^{D}_{L} := \ext^c(\cE,L).
\end{equation*}
We also write $\cE^D := (\cE)^{D}_{\O_X}$.\footnote{This differs slightly from the notation of \cite{HL}, where $\cE^D$ denotes $\ext^c(\cE,\omega_X)$.} In this paper we often use \cite[Prop.~1.1.6, 1.1.10]{HL}. Although these results are stated for smooth \emph{projective} varieties, they also hold for smooth varieties and even for regular Noetherian $\C$-schemes.\footnote{Indeed the only place where projectivity is used, is in the proof of the vanishing statement of Prop.~1.1.6(i), namely $\ext^{i}(\F,\O_X) = 0$ for all $i < c$ where $c$ is the codimension of the support of $\F$. However this vanishing also holds at the level of local rings using ``local duality'' \cite[Thm.~4.4]{Hun}, thereby avoiding projectivity.}

The topological Euler characteristic of a $\C$-stack is by definition the naive Euler characteristic of the associated set of isomorphism classes of $\C$-valued points as defined in \cite[Def.~4.8]{Joy1}.

\section{Consequences and reduction to affine case} \label{section1}

In the first subsection, we prove that Theorem \ref{main} implies Theorems \ref{cor} and \ref{thm2}. In the second subsection, we reduce Theorem \ref{main} to the affine case.

\subsection{Consequences} 

\begin{proof}[Proof of Theorem \ref{cor}]
Any reflexive sheaf $\ccR$ on a smooth threefold $X$ has homological dimension $\leq 1$ (\cite[Prop.~1.3]{Har2} and the Auslander-Buchsbaum formula). Therefore we have a resolution
$$
0 \rightarrow E_1 \rightarrow E_0 \rightarrow \ccR \rightarrow 0,
$$
where $E_0, E_1$ are locally free. Dualizing this short exact sequence and breaking up the resulting long exact sequence gives
\begin{align}
0 &\rightarrow \ccR^* \rightarrow E_0^* \rightarrow \cC \rightarrow 0, \label{firstses} \\
0 &\rightarrow \cC \rightarrow E_1^* \rightarrow \ext^1(\ccR,\O_X) \rightarrow 0, \label{secondses}
\end{align}
where $\cC$ is the image of $E_0^* \rightarrow E_1^*$. Dualizing \eqref{firstses} gives
$$
\ext^2(\cC,\O_X) \cong \ext^1(\ccR^*,\O_X).
$$
Dualizing \eqref{secondses} gives
$$
\ext^2(\cC,\O_X) \cong \ext^1(\ccR,\O_X)^D,
$$
where we recall that $\ext^1(\ccR,\O_X)$ is zero or 0-dimensional because $\ccR$ is reflexive \cite[Prop.~1.1.10]{HL}. Therefore we conclude
\begin{equation} \label{extD}
\ext^1(\ccR,\O_X)^D \cong \ext^1(\ccR^*,\O_X).
\end{equation}

Assume $\ext^1(\ccR,\O_X)$ is non-zero, because otherwise $\ccR$ is locally free and there is nothing to prove. Recall the definition of $P_{\ccR}(q)$ from the statement of Theorem \ref{cor} and let $d := \deg P_{\ccR}(q)$, which equals $\ell(\ext^1(\ccR,\O_X))$. We claim that there is an isomorphism
\begin{equation} \label{dualmap}
(\cdot)^D : \Quot_X(\ext^1(\ccR,\O_X),n) \rightarrow \Quot_X(\ext^1(\ccR,\O_X)^D,d-n).
\end{equation}
Indeed any short exact sequence
$$
0 \rightarrow \cK \rightarrow \ext^1(\ccR,\O_X) \rightarrow \cQ \rightarrow 0
$$
with $\ell(\cQ) = n$ dualizes to a short exact sequence
$$
0 \rightarrow \cQ^D \rightarrow \ext^1(\ccR,\O_X)^D \rightarrow \cK^D \rightarrow 0.
$$
Note that $\ell(\cQ^D) = \ell(\cQ)$ and $\cQ^{DD} \cong \cQ$ for any 0-dimensional sheaf $\cQ$ on $X$.\footnote{This argument can easily be extended to the level of flat families of 0-dimensional sheaves using derived duals and cohomology and base change for Ext groups \cite[Prop.~3.1]{Sch}.} The claim follows. The first statement of the theorem follows from the following computation
\begin{align*}
P_{\ccR}^*(q) &= q^d \sum_{n=0}^{\infty} e\big(\Quot_X(\ext^1(\ccR,\O_X),n)\big) \, (q^{-1})^n \\
&= \sum_{n=0}^{\infty} e\big(\Quot_X(\ext^1(\ccR,\O_X)^D,n)\big) \, q^n \\
&= \sum_{n=0}^{\infty} e\big(\Quot_X( \ext^1(\ccR^*,\O_X),n)\big) \, q^n = P_{\ccR^*}(q),
\end{align*}
where we first used \eqref{dualmap} and then \eqref{extD}. 

Finally we note that for any line bundle $L$ we have 
\begin{equation*}
P_{\ccR \otimes L}(q) = P_{\ccR}(q).
\end{equation*}
The second statement of the theorem follows from the fact that for any rank 2 reflexive sheaf $\ccR$ we have $\ccR^* \cong \ccR \otimes \det(\ccR)^{-1}$ \cite[Prop.~1.10]{Har2}.
\end{proof}

Before we prove that Theorem \ref{main} implies Theorem \ref{thm2}, we need two lemmas. 
\begin{lemma} \label{c3}
Let $\ccR$ be a reflexive sheaf on a smooth projective threefold $X$ with polarization $H$. Then 
\begin{enumerate}
\item[\rm (1)] $c_1(\ccR^*) = -c_1(\ccR)$, $c_2(\ccR^*) = c_2(\ccR)$, $c_3(\ccR)+c_3(\ccR^*) = 2 \ell(\ext^1(\ccR,\O_X))$. Moreover if $\ccR$ has rank 2, then $c_3(\ccR) = c_3(\ccR^{*}) = \ell(\ext^1(\ccR,\O_X))$.
\item[\rm (2)] If $\ccR$ is $\mu_H$-stable and reflexive, then so is $\ccR^*$.
\end{enumerate}
\end{lemma}

\begin{proof}
Part (2) is a standard result. For (1) we take a locally free resolution resolution
$$
0 \rightarrow E_1 \rightarrow E_0 \rightarrow \ccR \rightarrow 0,
$$
which exists as we saw in the previous proof. Dualizing gives 
$$
0 \rightarrow \ccR^* \rightarrow E_0^* \rightarrow E_1^* \rightarrow \ext^1(\ccR,\O_X) \rightarrow 0,
$$
where $\ext^1(\ccR,\O_X)$ is zero or 0-dimensional and we have $\ext^{>1}(\ccR,\O_X)=0$ by \cite[Prop.~1.1.10]{HL}. Taking $\ch(\cdot)$ of both exact sequences and using $c_i(E_j^*) = (-1)^i c_i(E_j)$ for $j=0,1$ gives the desired formulae. If $\ccR$ has rank 2, then $c_3(\ccR)$ equals the length of $\ext^1(\ccR,\O_X)$ by \cite[Prop.~2.6]{Har2}.\footnote{R.~Hartshorne's result is stated for $X = \PP^3$, but holds on any smooth projective threefold, e.g.~see \cite[Prop.~3.6]{GK1}.}
\end{proof}

\begin{lemma} \label{Kollar}
Let $X$ be a smooth projective threefold with polarization $H$. In $M_{X}^{H}(r,c_1,c_2,c_3)$ the locus $M_{X}^{H}(r,c_1,c_2,c_3)^\circ$ consisting of isomorphism classes $[\F]$ for which $\F^{**} / \F$ is zero or 0-dimensional is a (possibly empty) Zariski open subset. Moreover $M_{X}^{H}(r,c_1,c_2,c_3)^\circ$ is the locus of isomorphism classes $[\F]$ for which the singularity set of $\F$ is empty or 0-dimensional.
\end{lemma}
\begin{proof}
Denote the Hilbert polynomial determined by $X,H,r,c_1,c_2,c_3$ by $P(t)$. Suppose $[\F]$ is a closed point of $M_{X}^{H}(r,c_1,c_2,c_3)$ and $\cQ := \F^{**} / \F$ is zero or 0-dimensional. Then
\begin{equation} \label{lemma_eqn1}
P_{\F^{**}}(t) = P(t) + \frac{1}{2} c_3(\cQ),
\end{equation}
where $c_3(\cQ) \geq 0$ satisfies
\begin{equation} \label{c3eq}
c_3(\cQ) = c_3(\F^{**}) - c_3.
\end{equation}
By \cite[Prop.~3.6]{GK1} and the fact that $\F^{**}$ is rank $r$ $\mu_H$-stable with Chern classes $c_1,c_2$, we deduce that $c_3(\F^{**})$ is bounded above:
\begin{equation} \label{lemma_eqn3}
c_3(\F^{**}) \leq C(X,H,r,c_1,c_2),
\end{equation}
where $C:=C(X,H,r,c_1,c_2)$ is a constant only depending on $X,H,r,c_1,c_2$. We can choose this constant such that $C - c_3 \geq 0$ is even and we set $N:=\frac{1}{2}(C - c_3)$. From \eqref{lemma_eqn1}, \eqref{c3eq}, and \eqref{lemma_eqn3}, we conclude
\begin{equation} \label{inclMcirc}
M_{X}^{H}(r,c_1,c_2,c_3)^\circ  \subset \{[\F] \ : \ P_{\F^{**}}(t) \leq P(t)+N\}.
\end{equation}
Here for any two polynomials $p(t),q(t) \in \Q[t]$ the notation $p(t) \leq q(t)$ means $p(t) \leq q(t)$ for all $t \gg 0$. 

We claim that inclusion \eqref{inclMcirc} is an equality. Indeed if $\cQ$ is 1-dimensional, then 
$$
P_{\F^{**}}(t) = P(t) + P_{\cQ}(t) > P(t) + N,
$$
because $P_{\cQ}(t)$ is a degree 1 polynomial with positive leading coefficient. Zariski openness follows from the fact that the map $[\F] \mapsto P_{\F^{**}}(t)$ is upper semi-continuous with respect to $\leq$ by a result of J.~K\'ollar \cite[Prop.~28(3)]{Kol}.

Let $S(\F)$ denote the singularity set of a torsion free sheaf $\F$ on $X$. By \cite[Sect.~II.1.1]{OSS}, we have
$$
S(\F) = \mathrm{Supp}(\ext^1(\F,\O_X)) \cup \mathrm{Supp}(\ext^2(\F,\O_X)),
$$
where $\ext^1(\F,\O_X)$ has dimension $\leq 1$ and $\ext^2(\F,\O_X)$ has dimension $\leq 0$ by \cite[Prop.~1.1.10]{HL}. The second statement of the lemma easily follows from the fact that the singularity set of a reflexive sheaf on $X$ has dimension $\leq 0$.
\end{proof}

We are now ready to prove Theorem \ref{thm2}. We recall some notation and basic facts about constructible functions (e.g.~see \cite{Joy1} and references therein). Let $\Gamma$ be any abelian group. Suppose $\phi : X \rightarrow Y$ is a morphism of $\C$-schemes, locally of finite type, and $f : X(\mathbb{C}) \rightarrow \Gamma$ is a constructible function on the $\C$-valued points $X(\C)$ of $X$.  Then we define integration against the Euler characteristic measure by
$$
\int_X f \, de := \sum_{\kappa \in \Gamma} \kappa \cdot e\big( f^{-1}(\kappa) \big) \in \Gamma.
$$
The following map is a constructible function \cite[Prop.~3.8]{Joy1}
\begin{align*}
\phi_* f : Y(\C) &\rightarrow \Gamma \\
 y &\mapsto \int_{\phi^{-1}(y)} f|_{\phi^{-1}(y)} \, de.
\end{align*}
We will use the following key fact due to R.~MacPherson \cite{Mac} (see also \cite[Prop.~3.8]{Joy1})
\begin{equation} \label{Mac}
\int_X f \, de = \int_{Y} \phi_* f \, de.
\end{equation}
We use a slight generalization of this. Suppose $\phi : X \rightarrow Y$ is a set theoretic map between $\C$-schemes of finite type. We call $\phi$ a constructible morphism when $X$ can be written as a finite disjoint union $\bigsqcup_{i=1}^n X_i$  of locally closed subschemes $X_i \subset X$ of finite type such that $\phi|_{X_i} : X_i \rightarrow Y$ comes from a morphism of schemes. Then it is easy to see that $\phi_* f$ (defined as above) is a constructible function, for any constructible function $f : X(\C) \rightarrow \Gamma$, and \eqref{Mac} holds.

\begin{proof}[Proof of Theorem \ref{thm2}]
Fix $r,c_1,c_2$. Denote by $N_{X}^{H}(r,c_1,c_2,c_3)$ the moduli space of rank $r$ $\mu_H$-stable reflexive sheaves on $X$ with Chern classes $c_1,c_2,c_3$. By Lemma \ref{c3}, we can consider the following two maps
\begin{align*}
(\cdot)^{**} : \bigsqcup_{c_3} M_{X}^{H}(r,c_1,c_2,c_3)^\circ &\twoheadrightarrow \bigsqcup_{c_3} N_{X}^{H}(r,c_1,c_2,c_3), \\
(\cdot)^{*} : \bigsqcup_{c_3} N_{X}^{H}(r,c_1,c_2,c_3) &\twoheadrightarrow \bigsqcup_{c_3} N_{X}^{H}(r,-c_1,c_2,c_3).
\end{align*}
Restricted to $M_{X}^{H}(r,c_1,c_2,c_3)^\circ$ and $N_{X}^{H}(r,c_1,c_2,c_3)$, these are constructible morphisms.\footnote{In general, duals of members of a flat family do \emph{not} form a flat family.} The first map is surjective on closed points and the second map is bijective on closed points. At the level of closed points, the fibre of $(\cdot)^{**}$ over $[\mathcal{R}]$ is a union of Quot schemes
$$
\bigsqcup_{n=0}^{\infty} \Quot_X(\ccR,n).
$$

By \cite[Prop.~3.6]{GK1}, there exists a constant $C:=C(X,H,r,c_1,c_2)$ such that for any rank $r$ $\mu_H$-stable reflexive sheaf $\ccR$ on $X$ with Chern classes $c_1$, $c_2$ we have
$$
c_3(\ccR) \leq C.
$$
Likewise there exists a constant $C':=C(X,H,r,-c_1,c_2)$, such that for any rank $r$ $\mu_H$-stable reflexive sheaf $\ccR'$ on $X$ with Chern classes $-c_1$, $c_2$ we have
$$
c_3(\ccR') \leq C'.
$$
Using Lemma \ref{c3}, the Chern classes $c_3(\ccR)$ and $c_3(\ccR')$ are also bounded below:
$$
c_3(\ccR) = 2 \ell(\ext^1(\ccR,\O_X)) - c_3(\ccR^*) \geq - C'
$$
and likewise $c_3(\ccR') \geq - C$. We conclude that 
$$
\bigsqcup_{c_3} N_{X}^{H}(r,c_1,c_2,c_3), \ \bigsqcup_{c_3} N_{X}^{H}(r,-c_1,c_2,c_3)
$$
are of finite type. Define
\begin{align*}
f([\ccR]) &:= \sum_{n = 0}^{\infty} e\big( \Quot_X(\ccR,n) \big) \, q^{c_3(\ccR)-2n},  \\
g([\ccR]) &:= q^{c_3(\ccR)} P_{\ccR}(q^{-2}) \in \Z[q,q^{-1}], \\
P_{\ccR}(q) &:=\sum_{n = 0}^{\infty} e\big( \Quot_X(\ext^1(\ccR,\O_X),n) \big) \, q^{n},
\end{align*}
where the last sum is finite because $\ext^1(\ccR,\O_X)$ is zero or 0-dimensional and $P_{\ccR}(q)$ was introduced previously in Theorem \ref{cor}. 

Let $\phi := (\cdot)^{**}$ be the double dual map, then by \eqref{Mac}
\begin{align*}
\sum_{c_3} e\big( M_{X}^{H}(r,c_1,c_2,c_3)^\circ \big) \, q^{c_3} &= \sum_{c_3} q^{c_3} \int_{M_{X}^{H}(r,c_1,c_2,c_3)^\circ} 1 \, de \\
&= \sum_{c_3,c_3'} q^{c_3}  \int_{N_{X}^{H}(r,c_1,c_2,c_3')} \phi_*(1) \, de.
\end{align*}
Since the fibres of $\phi = (\cdot)^{**}$ are unions of Quot schemes $\Quot_X(\ccR,n)$, we find\footnote{Strictly speaking, $f$ is not a constructible function, because it is an infinite sum. However, it is constructible modulo $q^{-A}$ for \emph{arbitrary} $A >0$. Therefore the following equations  should be read modulo $q^{-A}$. Since $A > 0$ is arbitrary, the final equality, $\sum_{c_3} e\big( M_{X}^{H}(r,c_1,c_2,c_3)^\circ \big) \, q^{c_3}  =  \sfM(q^{-2})^{r e(X)} P_{r,c_1,c_2}(q)$, holds to all orders.}
\begin{align}
\begin{split} \label{fg}
\sum_{c_3} e\big( M_{X}^{H}(r,c_1,c_2,c_3)^\circ \big) \, q^{c_3} &= \sum_{c_3}  \int_{N_{X}^{H}(r,c_1,c_2,c_3)} f \, de \\
&= \sfM(q^{-2})^{r e(X)}  \sum_{c_3} \int_{N_{X}^{H}(r,c_1,c_2,c_3)} g \, de \\
&= \sfM(q^{-2})^{r e(X)} P_{r,c_1,c_2}(q),
\end{split}
\end{align}
where in the second line we use $f = M(q^{-2})^{re(X)} g$ (Theorem \ref{main}) and the third line is the definition of $P_{r,c_1,c_2}(q)$. Since the sums over $c_3$ on the right hand side of \eqref{fg} are finite, $P_{r,c_1,c_2}(q)$ is a Laurent polynomial. 

By Lemma \ref{c3} we have $c_3(\ccR)+c_3(\ccR^*) = 2 \ell(\ext^1(\ccR,\O_X))$. Therefore Theorem \ref{cor} gives
$$
q^{c_3(\ccR)} P_{\ccR}(q^{-2}) = q^{-c_3(\ccR^*)} P_{\ccR^*}(q^{2}).
$$
This translates into
$$
(\phi_* g)([\ccR]) = g([\ccR^*]) = g([\ccR]) \Big|_{q^{-1}}.
$$
Therefore 
\begin{align*}
P_{r,c_1,c_2}(q) &= \sum_{c_3} \int_{N_{X}^{H}(r,c_1,c_2,c_3)} g \, de \\
&= \sum_{c_3} \int_{N_{X}^{H}(r,-c_1,c_2,c_3)} \phi_* g \, de \\
&= \sum_{c_3} \int_{N_{X}^{H}(r,-c_1,c_2,c_3)} g \, de \Big|_{q^{-1}} = P_{r,-c_1,c_2}(q^{-1}). 
\end{align*}

Finally in the rank 2 case, we have $c_3(\ccR) = \ell(\ext^1(\ccR,\O_X))$. Therefore Theorem \ref{cor} implies
$$
q^{c_3(\ccR)} P_{\ccR}(q^{-2}) = q^{-c_3(\ccR)} P_{\ccR}(q^{2}),
$$
which gives $P_{2,c_1,c_2}(q) = P_{2,c_1,c_2}(q^{-1})$.
\end{proof}

\subsection{Reduction to affine case} \label{reductionsec}

In this section we first prove Theorem \ref{main} in the rather well-known case where $\F$ is locally free. Subsequently we reduce Theorem \ref{main} to the ``affine case''. 

When $X$ is a smooth variety and $Z \subset X$ is a closed subscheme, one can look at the formal neighbourhood of $Z \subset X$ which we denote by $\widehat{X}_Z$. On an open affine subset $U = \Spec A \subset X$, where $Z$ is given by the ideal $I$, we define the formal neighbourhood $\widehat{X}_Z$ by\footnote{Note that we define formal schemes by $\mathrm{Spec}$ of a ring rather than $\mathrm{Spf}$ of a ring.}
$$
\widehat{U}_{Z \cap U} = \Spec \varprojlim A /I^n.
$$
Since the map $\widehat{X}_Z \rightarrow X$ is flat, we can use $\widehat{X}_Z$ as part of an fpqc cover along which we can glue sheaves \cite[Tag 023T, Tag 03NV]{Sta}.

In this paper we will be using Quot schemes of 0-dimensional quotients of sheaves on several types of schemes. \extendspace{Originally Grothendieck's construction implies (in particular) that the Quot functor of} $0$-dimensional quotients is representable by a scheme $\Quot_X(\F,n)$ when $X$ is a projective $\C$-scheme and $\F$ is a coherent sheaf. This was extended to $X$ quasi-projective in \cite{AK}. T.~S.~Gustavsen, D.~Laksov, R.~M.~Skjelnes \cite{GLS} showed representability of the Quot functor of 0-dimensional quotients for $X = \Proj S$ (any graded $\C$-algebra $S$) or $X = \Spec A$ (any $\C$-algebra $A$) and $\F$ quasi-coherent. We need this level of generality because we will encounter the case $A = \C[\![x,y,z]\!]$, which is Noetherian but \emph{not of finite type}.

\begin{proposition} \label{locfree}
Let $X$ be a smooth quasi-projective threefold and let $\F$ be a rank $r$ locally free sheaf on $X$. Then
$$
\sum_{n=0}^{\infty} e\big(\Quot_X(\F,n)\big) \, q^n = \sfM(q)^{r e(X)}.
$$
\end{proposition}
\begin{proof}
Let $P$ be any closed point in $X$ and consider the punctual Quot scheme
$$
\Quot_X(\F,n)_0 \subset \Quot_X(\F,n),
$$
consisting of quotients $\F \twoheadrightarrow \cQ$ with $\mathrm{Supp} \, \cQ = \{P\}$.  We study the function
\begin{align*}
&g : \Z_{\geq 0} \rightarrow \Z, \\
&g(n) := e(\Quot_X(\F,n)_0),
\end{align*}
where $g(0):=1$. Since $\F$ is locally free, we have
\begin{align} \label{Quoteq1}
\Quot_X(\F,n)_0 \cong \Quot_X(\O_{X}^{\oplus r},n)_0.
\end{align}
There exists a torus action $T = \C^{*r}$ on $\Quot_{X}(\O_{X}^{\oplus r},n)_0$ defined by scaling the summands of $\O_{X}^{\oplus r}$. The fixed locus of this action is 
\begin{align} \label{Quoteq2}
\Quot_{X}(\O_{X}^{\oplus r},n)^T_0 \cong \bigsqcup_{n_1 + \cdots + n_r = n} \prod_{i=1}^{r} \Hilb^{n_i}(X)_0,
\end{align}
where $\Hilb^{n}(X)_0$ denotes the punctual Hilbert scheme of length $n$ subschemes of $X$ supported at $P$. Combining \eqref{Quoteq1} and \eqref{Quoteq2}, we obtain
\begin{align} \label{intermedQuot}
\sum_{n=0}^{\infty} g(n) \, q^n = \Bigg( \sum_{n=0}^{\infty} e(\Hilb^n(X)_0) \, q^n \Bigg)^{r} = \sfM(q)^r,
\end{align}
where the second equality follows from a result of J.~Cheah \cite{Che}.\footnote{This can be seen by noting that $e(\Hilb^n(X)_0) = e(\Hilb^n(\Spec \C[[x,y,z]])$. Using the standard $\C^{*3}$-action on $\Spec \C[[x,y,z]]$, this is the number of monomial ideals in $x,y,z$ defining a 0-dimensional scheme of length $n$.} 

Next we consider the Hilbert-Chow type morphism
\begin{align*}
&\mathrm{HC} : \Quot_X(\F,n) \rightarrow \Sym^n(X), \\ 
&[\F \twoheadrightarrow \cQ] \mapsto \sum_{i} n_i P_i,
\end{align*}
where $P_i$ are the $\C$-valued points of the support of $\cQ$ and $n_i$ is the length of $\cQ$ at $P_i$. Consider the constructible function
\begin{align*}
&G : \Sym^n(X) \rightarrow \Z, \\
&G\big(\sum_{i} n_i P_i\big) := e\big(\mathrm{HC}^{-1}\big(\sum_{i} n_i P_i\big)\big) = \prod_i g(n_i).
\end{align*}
Then
$$
\sum_{n=0}^{\infty} e\big(\Quot_X(\F,n)\big) \, q^n = \sum_{n=0}^{\infty} q^n \int_{\Sym^n(X)} G \, de = \Bigg( \sum_{n=0}^{\infty} g(n) \, q^n \Bigg)^{e(X)} = \sfM(q)^{r e(X)},
$$
where the second equality follows from  \cite[App.~A.2]{BK} and the third equality follows from \eqref{intermedQuot}.
\end{proof}
This proposition implies Theorem \ref{main} when $\F$ is locally free. We now reduce Theorem \ref{main} to the affine case. 
\begin{proposition} \label{redtolocal}
Let $X$ be a smooth projective threefold and let $\F$ be a rank $r$ torsion free sheaf on $X$ of homological dimension $\leq 1$. We assume the following: if $\Spec A$ is an affine scheme of one of the following types
\begin{itemize}
\item $\Spec A$ a Zariski open subset of $X$,
\item $\Spec A = \Spec \widehat{\O}_{X,P}$, where $\widehat{\O}_{X,P}$ is the completion of the stalk $\O_{X,P}$ at a closed point $P \in X$,
\end{itemize}
then
\begin{align} 
\label{affinemaineq}
\sum_{n=0}^{\infty} e\big(\Quot_{A}(\F|_{\Spec A},n)\big) \, q^n = \sfM(q)^{r e(\Spec A)} \sum_{n=0}^{\infty} e\big(\Quot_{A}(\ext^1(\F|_{\Spec A},\O_{\Spec A}),n)\big) \, q^n,
\end{align}
where $e(\Spec A) = 1$ for $A = \widehat{\O}_{X,P}$. Then Theorem \ref{main} is true for $X$ and $\F$.
\end{proposition}
\begin{proof}
Take $X$ and $\F$ as in the proposition.  Assume \eqref{affinemaineq} is true for any $\Spec A$ as described in the proposition. We will show that the formula of Theorem \ref{main} is true for $X$ and $\F$. 

Let $S \subset X$ be the scheme-theoretic support of $\ext^1(\F,\O_X)$. Then $S$ has dimension $\leq 1$ by \cite[Prop.~1.1.10]{HL}. Let $C$ be the union of the 1-dimensional connected components of $S$ and let $Z$ be the union of the 0-dimensional connected components of $S$. We note that $\F|_{X \setminus S}$ is locally free (by \cite[Ch.~II]{OSS}) and $\F|_{X \setminus C}$ is reflexive by \cite[Prop.~1.1.10]{HL}. 

There exists a closed subset $H \subset X$ such that $C \cap H$ is 0-dimensional, $H \cap Z = \varnothing$, and the complement $U$ of $H$ is affine. This can be seen by embedding $X$ in a projective space and intersecting with a general hyperplane. Write $U = \Spec A \subset X$. Let $P_1, \ldots, P_\ell$ be the closed points of $C \cap H$. Define
\begin{align*}
X^{\circ} := X \setminus C, \ H^{\circ} := H \setminus (C \cap H).
\end{align*}
We take the following fpqc cover of $X$
\begin{align*}
U=\Spec A, \ \widehat{X^{\circ}}_{H^\circ}, \ \widehat{X}_{\{P_i\}} = \Spec \widehat{\O}_{X,P_i} \qquad \forall i=1, \ldots \ell.
\end{align*}
By fpqc descent and the fact that we are considering 0-dimensional quotients (so there are no gluing conditions on overlaps), we obtain a geometrically bijective constructible morphism\footnote{The definition of a constructible morphism was given before the proof of Thm.~1.3. Such a map is called geometrically bijective, when it  is a bijection on $\C$-valued points. If $\phi : X \rightarrow Y$ is a geometrically bijective constructible morphism, then $e(X) = e(Y)$ by \eqref{Mac}.} from $\Quot_X(\F,n)$ to
\begin{align*}
\bigsqcup_{a+b+n_1+\cdots+n_\ell= n} \!\!\!\!\!\!\!\!\!\!\!\! \Quot_U(\F|_U,a) \times \Quot_{\widehat{X^\circ}_{H^\circ}}(\F|_{\widehat{X^\circ}_{H^\circ}},b) \times \prod_{i=1}^{\ell} \Quot_{\widehat{X}_{\{P_i\}}}(\F|_{\widehat{X}_{\{P_i\}}},n_i).
\end{align*}
We obtain
\begin{align}\label{mess}
\sum_{n=0}^{\infty} e\big(\Quot_X(\F,n)\big) \, q^n&= \sum_{n=0}^{\infty} e\big(\Quot_{A}(\F|_{\Spec A},n)\big) \, q^n \cdot \sum_{n=0}^{\infty} e\big(\Quot_{\widehat{X^\circ}_{H^\circ}}(\F|_{\widehat{X^\circ}_{H^\circ}},n)\big) \, q^n\nonumber \\
&\cdot \prod_{i=1}^{\ell} \sum_{n=0}^{\infty} e\big(\Quot_{\widehat{X}_{\{P_i\}}}(\F|_{\widehat{X}_{\{P_i\}}},n)\big) \, q^n.
\end{align}
The expressions
$$
\sum_{n=0}^{\infty} e\big(\Quot_{A}(\F|_{A},n)\big), \ \sum_{n=0}^{\infty} e\big(\Quot_{\widehat{X}_{\{P_i\}}}(\F|_{\widehat{X}_{\{P_i\}}},n)\big) \, q^n
$$
are known because we assume \eqref{affinemaineq} is true. Now $\F|_{\widehat{X^\circ}_{H^\circ}}$ is locally free because $X^\circ \cap S = Z$ and $H^\circ \cap S = \varnothing$. Moreover Proposition \ref{locfree} implies\footnote{As formulated, Proposition \ref{locfree} only applies to a smooth quasi-projective threefold. However the same proof works in the present setting. The essential point is that for any $\C$-valued point $P$ of $\widehat{X^\circ}_{H^\circ}$, the formal neighborhood of $P$ in $\widehat{X^\circ}_{H^\circ}$ is $\mathrm{Spec}\,\C[[x,y,z]]$ by \cite[(7.8.3)]{EGA} and \cite[TAG 07NU, TAG 0323]{Sta}. The Hilbert-Chow morphism is constructed at the level of generality we need by D.~Rydh \cite{Ryd}.}
$$
\sum_{n=0}^{\infty} e\big(\Quot_{\widehat{X^\circ}_{H^\circ}}(\F|_{\widehat{X^\circ}_{H^\circ}},n)\big) \, q^n = \sfM(q)^{r e(\widehat{X^\circ}_{H^\circ})}.
$$
We conclude that \eqref{mess} is equal to
\begin{align*}
&\sfM(q)^{r (e(\Spec A) + e(\widehat{X^\circ}_{H^\circ}) +\sum_{i=1}^{\ell} e(\Spec \widehat{\O}_{X,P_i}))} \sum_{n=0}^{\infty} e\big(\Quot_{A}(\ext^1(\F|_{A},\O_{\Spec A}),n)\big) \, q^n  \\
&\qquad\qquad\qquad\qquad\qquad\qquad \cdot \prod_{i=1}^{\ell} \sum_{n=0}^{\infty} e\big(\Quot_{\widehat{X}_{\{P_i\}}}(\ext^1(\F|_{\widehat{X}_{\{P_i\}}}, \O_{\widehat{X}_{\{P_i\}}}),n)\big) \, q^n \\
&= \sfM(q)^{r (e(\Spec A) + e(\widehat{X^\circ}_{H^\circ}) +\sum_{i=1}^{\ell} e(\Spec \widehat{\O}_{X,P_i}))} \sum_{n=0}^{\infty} e\big(\Quot_X(\ext^1(\F,\O_X),n)\big) \, q^n. 
\end{align*}
The proposition follows from 
\begin{equation*}
e(\Spec A) + e(\widehat{X^\circ}_{H^\circ}) + \sum_{i=1}^{\ell} e(\Spec \widehat{\O}_{X,P_i}) = e(\Spec A) + e(H^\circ) + \ell = e(X).
\end{equation*}
Here we used that $\widehat{X^\circ}_{H^\circ}$ and $H^\circ$ have homeomorphic subspaces of $\C$-valued points and each $\Spec \widehat{\O}_{X,P_i}$ has a single $\C$-valued point (Lemma \ref{technical} below). 
\end{proof}

\begin{lemma} \label{technical}
Let $A$ be a finitely generated $\C$-algebra and $I \subsetneq A$ an ideal. Denote by $\widehat{A}$ the formal completion of $A$ with respect to $I$. The subspace of closed points of $\Spec A/I$ is homeomorphic with the subspace of closed points of $\Spec \widehat{A}$. Moreover the subspace of closed points of $\Spec \widehat{A}$ equals the subspace of $\C$-valued points of $\Spec \widehat{A}$.
\end{lemma}
\begin{proof}
The surjection $\widehat{A} \twoheadrightarrow A/I$ induces a homeomorphism from $\Spec A/I$ onto its image in $\Spec \widehat{A}$. This map sends a prime ideal $\mathfrak{q}$ containing $I$ to $\widehat{\mathfrak{q}}$. Furthermore this map sends a maximal ideal $\mathfrak{m}$ containing $I$ to the maximal ideal $\widehat{\mathfrak{m}}$ and all maximal ideals of $\widehat{A}$ arise in this way \cite[Ch.~III.3.4, Prop.~8]{Bou1}. 

Any closed point $\widehat{\mathfrak{m}} \subset \widehat{A}$ is a $\C$-valued point because $\widehat{A} / \widehat{\mathfrak{m}} \cong \widehat{A / \mathfrak{m}} \cong \C$. Conversely suppose $\mathfrak{p} \subset \widehat{A}$ is a $\C$-valued point. The surjection $A \twoheadrightarrow A/I$ factors as
$$
A \hookrightarrow \widehat{A} \twoheadrightarrow A/I,
$$
where the first map is injective by Krull's theorem \cite[Cor.~10.18]{AM}. The ideal $\mathfrak{q} := \mathfrak{p} \cap A$ is a $\C$-valued point and therefore corresponds to a maximal ideal ($A$ is finitely generated). We claim $I \subset \mathfrak{q}$ and $\mathfrak{p} = \widehat{\mathfrak{q}}$, which implies $\mathfrak{p}$ is maximal by the first part of the proof. 

By the first part of the proof, there exists a maximal ideal $\mathfrak{m}$ containing $I$ such that $\mathfrak{p} \subset \widehat{\mathfrak{m}}$ (any prime ideal lies in a maximal ideal). Therefore
$$
\mathfrak{q} = \mathfrak{p} \cap A \subset \widehat{\mathfrak{m}} \cap A = \mathfrak{m}.
$$
Since $\mathfrak{q}$ is maximal, we have $\mathfrak{q} = \mathfrak{m}$, so $\mathfrak{q}$ contains $I$ and $\mathfrak{p} \subset \widehat{\mathfrak{q}}$. By \cite[Prop.~1.17, 10.13]{AM}, we  obtain the other inclusion
$$
\mathfrak{p} \supset \mathfrak{q} \otimes_A \widehat{A} = \widehat{\mathfrak{q}}. 
$$
This proves the claim. We conclude that the collection of $\C$-valued points of $\widehat{A}$ and the collection of closed points of $\widehat{A}$ coincide.
\end{proof}

\section{Pandharipande-Thomas pairs and Quot schemes} \label{section2}

In this section we assume $A$ is a 3-dimensional Noetherian regular $\C$-algebra and $M$ is an $A$-module of homological dimension 1. Later we will be interested in the case $A, M$ arise from $X, \F$ as in Proposition~\ref{redtolocal}.

Denote the category of quasi-coherent sheaves on $\Spec A$ by $\textrm{QCoh}(\Spec A)$ and by $A\textrm{-Mod}$ the category of $A$-modules. The global section functor
$$
\Gamma(\Spec A,\cdot) : \textrm{QCoh}(\Spec A) \rightarrow A\textrm{-Mod}
$$
is an equivalence of categories and we denote the inverse by $\widetilde{(\cdot)}$ as in \cite[Sect.~II.5]{Har1}. We will mostly work in the latter category $A\textrm{-Mod}$. For instance $\ext^i(\widetilde{M}, \O_{\Spec A})$ corresponds to $\Ext^i(M,A)$ by
$$
\Ext^i(M,A) = \Gamma(\Spec A, \ext^i(\widetilde{M}, \O_{\Spec A})).
$$

We give a very brief outline of this section. Denote by $\T_A$ the stack of 0-dimensional finitely generated $A$-modules.\footnote{The stack $\T_A$ also contains the zero module.} We first construct a Cohen-Macaulay curve $C \subset \Spec A$, with ideal $I_C \subset A$, and an effective divisor $S \subset \Spec A$ related to $M$. It turns out that we can use $C,S$ to construct injections
$$
\Ext^2(Q,M(S)) \hookrightarrow \Ext^2(Q,I_C) \cong \Ext^1(Q,A/I_C),
$$
for all $Q \in \T_A$. By varying $Q$ over the stack $\T_A$, we obtain a closed locus $\Sigma$ inside the \emph{moduli space of Pandharipande-Thomas pairs} on $C$. The main result of this section is Theorem \ref{key}, which establishes a geometric bijection\footnote{A geometric bijection is a morphism of schemes, which is a bijection on $\C$-valued points.} from  
$$
\Quot_A(\Ext^1(M,A)) := \bigsqcup_{n=0}^{\infty} \Quot_A(\Ext^1(M,A),n)
$$
onto the locus $\Sigma$. In Section \ref{Hall}, the locus $\Sigma$ naturally arises from a Hall algebra calculation involving
\begin{align*}
\Quot_A(M) &:= \bigsqcup_{n=0}^{\infty} \Quot_A(M,n). 
\end{align*}
The Hall algebra calculation of Section \ref{Hall} together with the geometric bijection of Theorem \ref{key} will lead to the proof of \eqref{affinemaineq} and therefore Theorem \ref{main}.

\subsection{Construction of local auxiliary curve}

Our key technical tool is the following theorem from Bourbaki \cite[Ch.~VII.4.9, Thm.~6, p.~270]{Bou2}.
\begin{theorem}[Bourbaki] \label{Bourbaki}
Let $A$ be a Noetherian integrally closed ring and $M$ a finitely generated rank $r$ torsion free $A$-module. Then there exists an $A$-module homomorphism $M \rightarrow A$ with free kernel.
\end{theorem}
Therefore there exists an ideal $I \subset A$ and a short exact sequence
\begin{equation} \label{cosection1}
0 \rightarrow A^{\oplus (r-1)} \rightarrow M \rightarrow I \rightarrow 0.
\end{equation}
Note that $I \neq A$ because $M$ is not locally free by assumption ($M$ has homological dimension 1). Although our constructions will depend on the cosection $M \rightarrow A$, our final result, equation \eqref{affinemaineq}, will not.\footnote{In \cite{GK2} we assumed the existence of a global version of this cosection. For rank 2 reflexive sheaves, this global cosection is precisely the data featuring in Hartshorne's version of the Serre correspondence \cite{Har2}. This led to the assumptions made in \cite{GK2} and explained in the introduction.} We start with a lemma.
\begin{lemma} \label{lem0}
In \eqref{cosection1}, $I=I_C(-S)$, where $S \subset \Spec A$ is a (possibly zero) effective divisor and $C \subset \Spec A$ is a non-empty Cohen-Macaulay curve.
\end{lemma}
\begin{proof}
There exists an effective divisor $S \subset \Spec A$ and a closed subscheme $C \subset \Spec A$ of dimension $\leq 1$ such that
$$
I \cong I_C(-S).
$$
This follows by embedding $I \hookrightarrow I^{**}$ and observing that $I^{**} / I$ has dimension~$\leq 1$ and that $I^{**} \subset A$ is a line bundle. 

Next we show that $C$ is not empty or 0-dimensional. If it were, then according to  \cite[Prop.~1.1.6]{HL} (see Notation in Section \ref{intro}), we would get
$\Ext^1(I_C(-S),A)  \cong \Ext^2((A/I_C)(-S),A) = 0$. Therefore \eqref{cosection1} induces a short exact sequence
$$
0 \rightarrow I^* \rightarrow M^* \rightarrow A^{\oplus(r-1)} \rightarrow 0
$$
and $\Ext^1(M,A) = 0$. This short exact sequence implies that $M^*$ is locally free, because $I^*$ is rank 1 reflexive \cite[Cor.~1.2]{Har2} and hence locally free  \cite[Prop.~1.9]{Har2}. The vanishing $\Ext^1(M,A) = 0$ implies that $M$ is reflexive by \cite[Prop.~1.1.10]{HL}. Hence $M \cong M^{**}$ is locally free contradicting the assumption that $M$ has homological dimension 1.

Finally we show that the 1-dimensional subscheme $C \subset X$ is Cohen-Macaulay. Indeed \eqref{cosection1} implies
$$
\Ext^2(I_C(-S),A) \cong \Ext^2(M,A) = 0,
$$
where the last equality follows from the fact that $M$ has homological dimension 1. Therefore
$$\Ext^3(A/I_C,A) \cong \Ext^2(I_C,A) = 0$$
and $C$ is Cohen-Macaulay by \cite[Prop.~1.1.10]{HL}.
\end{proof}

From now on, we will use the following version of \eqref{cosection1} 
\begin{equation} \label{cosection2}
0 \rightarrow A(S)^{\oplus (r-1)} \rightarrow M(S) \rightarrow I_C \rightarrow 0,
\end{equation}
where $S \subset X$ is the effective divisor and $\varnothing \neq C \subset X$ is the Cohen-Macaulay curve of Lemma \ref{lem0}. We are interested in
$$
\Quot_A(M) := \bigsqcup_{n=0}^{\infty} \Quot_A(M,n).
$$
We can write this as 
$$
\Quot_A(M) = \bigsqcup_{Q \in \T_A} \Hom(M,Q)^{\onto} / \sim,
$$
where $\Hom(M,Q)^{\onto}$ is the set of surjective $A$-module homomorphisms, and $\sim$ is the equivalence relation induced by automorphisms of $Q$. In Section \ref{Hall}, we will perform a Hall algebra calculation, which relates this to 
$$
\bigsqcup_{Q \in \T_A} \Ext^2(Q,M(S))^{\pure} / \sim.
$$
The goals of this section are firstly to define this locus and secondly to give a nice geometric characterization of it (Theorem \ref{key}). We start with a lemma.

\begin{lemma} \label{lem2}
For any $Q \in \T_A$, the map \eqref{cosection2} induces an inclusion
$$
\Ext^2(Q,M(S)) \hookrightarrow \Ext^2(Q,I_C).
$$
\end{lemma}
\begin{proof}
Applying $\Hom(Q,\cdot)$ to \eqref{cosection2} gives
$$
\cdots \longrightarrow \Ext^2(Q,A(S)^{\oplus (r-1)}) \longrightarrow \Ext^2(Q,M(S)) \longrightarrow \Ext^2(Q,I_C) \longrightarrow \cdots.
$$
The statement follows from
\begin{align*}
\Ext^2(Q,A(S)^{\oplus (r-1)}) &\cong \Ext^2(Q,A)^{\oplus (r-1)} \otimes A(S), \\ 
\Ext^2(Q,A) &= 0,
\end{align*}
where in the second line we use that $Q$ is 0-dimensional \cite[Prop.~1.1.6]{HL}.
\end{proof}

From the previous lemma, we obtain an inclusion
\begin{equation} \label{inducedincl}
\Ext^2(Q,M(S)) \hookrightarrow \Ext^2(Q,I_C),
\end{equation}
for all $Q \in \T_A$. 
By the short exact sequence
$
0 \rightarrow I_C \rightarrow A \rightarrow A / I_C \rightarrow 0
$
and the fact that $Q$ is 0-dimensional, we see that 
\begin{equation} \label{twice}
\Ext^2(Q,I_C) \cong \Ext^1(Q,A / I_C).
\end{equation}
The elements of $\Ext^1(Q,A / I_C)$ correspond to short exact sequences
$$
0 \rightarrow A / I_C \rightarrow F \rightarrow Q \rightarrow 0,
$$
where $F$ is a 1-dimensional $A$-module. Such an extension is known as a \emph{Pandharipande-Thomas pair} on $\Spec A$ whenever $F$ is a pure $A$-module \cite{PT1}. We denote the locus of PT pairs by $\Ext^1(Q,A / I_C)^{\pure}$. Put differently, a PT pair on $\Spec A$ consists of $(F,s)$ where
\begin{itemize}
\item $F$ is a pure dimension 1 $A$-module, 
\item $s : A \rightarrow F$ is an $A$-module homomorphism with 0-dimensional cokernel.
\end{itemize}
See \cite{PT1} for details. Let $\PT_A(C)$ be the moduli space of PT pairs $(F,s)$ on $\Spec A$ where $F$ has scheme theoretic support $C$. We can write this as 
\begin{align*}
\PT_A(C) = \bigsqcup_{Q \in \T_A} \Ext^1(Q,A/I_C)^{\pure} / \sim, 
\end{align*}
where $\sim$ denotes the equivalence relation induced by automorphisms of $Q$. Using inclusion \eqref{inducedincl} and \eqref{twice}, we define  
\begin{align}
\begin{split} \label{locus}
\Sigma :=\,&\bigsqcup_{Q \in \T_A} \Ext^2(Q,M(S))^{\pure} / \sim, \\
\Ext^2(Q,M(S))^{\pure} :=\,&\Ext^2(Q,M(S)) \cap \Ext^1(Q,A / I_C)^{\pure}.
\end{split}
\end{align}
This defines a closed subset $\Sigma \subset \PT_A(C)$. The locus $\Sigma$ depends on the choice of cosection \eqref{cosection2}. In the following theorem, we give a nice geometric characterization of $\Sigma$. Later, when taking Euler characteristics in Section \ref{Hallfinal}, it leads to a proof of equation \eqref{affinemaineq}.
\begin{theorem} \label{key}
Let $A$ be a 3-dimensional regular Noetherian $\C$-algebra and $M$ a torsion free $A$-module of homological dimension 1. Fix a cosection \eqref{cosection2}. Then there exists a morphism of $\C$-schemes
$$
\Quot_A(\Ext^1(M,A)) \rightarrow \PT_A(C),
$$
which induces a bijection between the $\C$-valued points of $\Quot_A(\Ext^1(M,A))$ and the $\C$-valued points of the locus $\Sigma$ defined in \eqref{locus}.
\end{theorem}

The next two subsections are devoted to the proof of this theorem.

\subsection{The rank 2 case} \label{rank2}

We first prove Theorem \ref{key} for $r=2$. This case is similar to the main result of \cite{GK2}, but with some minor modifications. We also show why the map of Theorem \ref{key} is a morphism of schemes, which was not discussed in detail in \cite{GK2}.

\begin{proof}[Proof of Theorem \ref{key} for $r=2$]
We first construct a set theoretic map at the level of $\C$-valued points from $\Sigma$ to $\Quot_A(\Ext^1(M,A))$ and show it is a bijection. In Step 4 we show that the inverse map is a morphism of schemes. 

\bigskip

\noindent \textbf{Step 1:} In this step we give a characterization of the image of $\Ext^1(Q[-1],M(S))$ of the inclusion of Lemma~\ref{lem2} 
$$
\Ext^1(Q[-1],M(S)) \hookrightarrow \Ext^1(Q[-1],I_C).
$$
An element of the Ext group $\Ext^1(Q[-1],M(S))$ corresponds to an exact triangle 
$$
M(S) \rightarrow J^\mdot \rightarrow Q[-1].
$$
The image of this element in $\Ext^1(Q[-1],I_C)$ is the third column of the following diagram, where all rows and columns are exact triangles, all squares commute, and the third column is given by the octahedral axiom:
\begin{displaymath}
\xymatrix
{
A(S) \ar[r] \ar@{=}[d] & M(S) \ar[r] \ar[d] & I_C \ar@{-->}[d] \\
A(S) \ar[r] & J^\mdot \ar[r] \ar[d] & I^\mdot \ar@{-->}[d] \\
& Q[-1] \ar@{=}[r] & Q[-1].
}
\end{displaymath}
An element $I_C \rightarrow I^\mdot \rightarrow Q[-1]$ of $\Ext^1(Q[-1],I_C)$ lies in the image of the inclusion of Lemma \ref{lem2} if and only if there exists a map $I^\mdot \rightarrow A(S)[1]$ such that the following diagram commutes
\begin{displaymath}
\xymatrix
{
I_C \ar[r] \ar[dr] & I^\mdot \ar@{-->}^<<<<{\exists}[d] \\
& A(S)[1],
}
\end{displaymath}
where the map $I_C \rightarrow A(S)[1]$ comes from \eqref{cosection2}. 

This leads to the following characterization of the image of the inclusion of Lemma \ref{lem2}: Suppose we are given an exact triangle $I_C \rightarrow I^\mdot \rightarrow Q[-1]$. It induces an exact sequence 
$$
\Ext^1(I^\mdot, A(S)) \rightarrow \Ext^1(I_C,A(S)) \rightarrow \Ext^3(Q,A(S)).
$$ 
Then $I_C \rightarrow I^\mdot \rightarrow Q[-1]$ lies in the image of the inclusion of Lemma \ref{lem2} if and only if
$$
\xi \in \mathrm{im}\Big( \Ext^1(I^\mdot, A(S)) \rightarrow \Ext^1(I_C,A(S)) \Big),
$$
where $\xi \in \Ext^1(I_C,A(S))$ is the extension determined by the cosection \eqref{cosection2}. Denote by
$$
s_\xi : A \longrightarrow \Ext^1(I_C,A(S)) 
$$
the $A$-module homomorphism that sends $1$ to $\xi$. Consequently an element $I_C \rightarrow I^\mdot \rightarrow Q[-1]$ lies in the image of inclusion of Lemma \ref{lem2} if and only if the composition
$$
A \stackrel{s_\xi}{\longrightarrow} \Ext^1(I_C,A(S)) \longrightarrow \Ext^3(Q,A(S))
$$
is zero. The cokernel of $s_\xi$ can be computed by applying $\Hom(\cdot, A(S))$ to the short exact sequence \eqref{cosection2}
$$
0 \rightarrow A(S) \rightarrow M^* \rightarrow A \stackrel{s_\xi}{\rightarrow} \Ext^1(I_C,A(S)) \rightarrow \Ext^1(M,A) \rightarrow 0, 
$$
where we used $\Hom(I_C,A(S)) \cong A(S)$ because $C$ has codimension 2 \cite[Prop.~1.1.6]{HL}. Also note that 
$$
\Ext^1(I_C,A(S)) \cong \Ext^2(A / I_C, A(S)) = (A / I_C)^D_{A(S)},
$$
using Notation of Section \ref{intro}. By \cite[Thm.~21.5]{Eis}, this is isomorphic to 
$$
\omega_C \otimes \omega_{A}^{-1}(S), 
$$
which is pure 1-dimensional by \cite[Sect.~1.1]{HL}.\footnote{As an aside, we note that if $\Ext^1(M,A)$ is zero or 0-dimensional, then $((A/I_C)^D_{A(S)}, s_\xi)$ is a PT pair. We will not use this.} 

We conclude that an element $I_C \rightarrow I^\mdot \rightarrow Q[-1]$ lies in the image of the inclusion of Lemma \ref{lem2} if and only if there exists an $A$-module homomorphism 
\begin{equation} \label{factor}
\begin{gathered}
\xymatrix
{
\Ext^1(I_C,A(S)) \ar[dr] \ar[r] & \Ext^1(M,A) \ar@{-->}^<<<<{\exists}[d] \\
& \Ext^3(Q,A(S)),
}
\end{gathered}
\end{equation}
such that the triangle in the diagram commutes, where 
$$
\Ext^3(Q,A(S)) =: Q^{D}_{A(S)}. 
$$

\noindent \textbf{Step 2:} In this step we show that an element in the image of the inclusion
$$
\Ext^1(Q[-1],M(S)) \subset \Ext^1(Q[-1],I_C) \cong \Ext^1(Q,A/I_C)
$$ 
lies in $\Ext^1(Q,A/I_C)^{\pure}$ if and only if the map in \eqref{factor} is surjective. Fix an element of $\Ext^1(Q[-1],I_C)$
\begin{equation*} 
I_C \rightarrow I^\mdot \rightarrow Q[-1]
\end{equation*}
and consider the induced exact sequence
$$
\cdots \longrightarrow \Ext^1(I_C,A(S)) \longrightarrow \Ext^3(Q,A(S)) \longrightarrow \Ext^2(I^\mdot, A(S)) \longrightarrow 0.
$$
Here we use that $\Ext^2(I_C,A) \cong \Ext^3(A/I_C,A)=0$, because $C$ is Cohen-Macaulay (Lemma \ref{lem0} and \cite[Prop.~1.1.10]{HL}). 

Next we write $I^\mdot = \{A \rightarrow F\}$. This is possible because $\Ext^1(Q[-1],I_C) \cong \Ext^1(Q,A/I_C)$ (by \eqref{twice}), so any exact triangle $I_C \rightarrow I^\mdot \rightarrow Q[-1]$ is uniquely determined by an extension $0 \rightarrow A/I_C \rightarrow F \rightarrow Q \rightarrow 0$. Therefore we have an exact triangle $I^\mdot \rightarrow A \rightarrow F$. Dualizing induces an isomorphism
$$
\Ext^2(I^\mdot,A) \cong \Ext^3(F,A).
$$
Moreover $F$ is pure if and only if $\Ext^3(F,A) = 0$ \cite[Prop.~1.1.10]{HL}. The claim follows because $\Ext^3(F,A) = 0$ if and only if 
$$
\Ext^2(I^\mdot,A(S)) \cong \Ext^2(I^\mdot,A) \otimes A(S) = 0. 
$$

\noindent \textbf{Step 3:} In this step we construct a set theoretic bijective map at the level of $\C$-valued points
\begin{equation} \label{setmap}
\Sigma \subset \PT_A(C) \rightarrow \Quot_A(\Ext^1(M,A)),
\end{equation}
where $\Sigma$ was defined in \eqref{locus}. 

For any 0-dimensional $A$-module $Q$ and any pure dimension 1 $A$-module $F$ we have \cite[Prop.~1.1.10]{HL}
\begin{equation} \label{doubleD}
Q^{DD} \cong Q, \ F^{DD} \cong F.
\end{equation}
More generally a PT pair 
$$
0 \rightarrow A / I_C \rightarrow F \rightarrow Q \rightarrow 0
$$
dualizes to a short exact sequence
$$
0 \rightarrow F^D \rightarrow (A / I_C)^D \rightarrow Q^D \rightarrow 0,
$$
where $F^D$ is pure and $Q^D$ is 0-dimensional.

We summarize the results of Steps 1 and 2. Given a PT pair $I^\mdot = \{A \rightarrow F\}$ with cokernel $Q$, we can form the following diagram
\begin{equation} \label{Amindiag0}
\begin{gathered}
\xymatrix
{
& & A \ar^>>>>>{s_\xi}[d] & & \\
0 \ar[r] & F^{D}_{A(S)} \ar[r] & (A/I_C)^{D}_{A(S)} \ar[r] \ar[d] & Q^{D}_{A(S)} \ar[r] \ar@{=}[d] & 0 \\
& & \Ext^1(M,A) \ar[d] \ar@{-->>}^<<<<<{\exists}[r] & Q^{D}_{A(S)}, & \\
& & 0 & &
}
\end{gathered}
\end{equation}
where $((A / I_C)^D_{A(S)}, s_\xi)$ was constructed from \eqref{cosection2} in Step 1. Then $I^\mdot$ lies in the image of the injection of Lemma \ref{lem2} if and only if the indicated surjection exists. We have produced the map \eqref{setmap} at the level of sets.

Applying $(\cdot)^{D}_{A(S)}$ to the middle row of diagram \eqref{Amindiag0} gives back the original PT pair $I^\mdot = \{A \rightarrow F\}$ by \eqref{doubleD}. From this fact, it is easy to construct the inverse of the map described above by starting from a surjection $\Ext^1(M,A) \twoheadrightarrow Q$ and inducing the middle row of diagram \eqref{Amindiag0}. 

\bigskip

\noindent \textbf{Step 4:} Finally we show that the inverse of the set theoretic bijection \eqref{setmap} constructed in the previous step is a morphism of schemes. Let $U = \Spec A$ and let $C \subset U$ be the Cohen-Macaulay curve of Lemma \ref{lem0}. Since we work in the affine setting, we can use module notation as we have been doing so far, but for this step we prefer sheaf notation. Denote by $\F$ the torsion free sheaf corresponding to $M$ and let $B$ be any base $\C$-scheme of finite type. Denote projection by $\pi_U : U \times B \rightarrow U$. Suppose we are given a $B$-flat family of 0-dimensional quotients
$$
\pi_U^* \ext^1(\F,\O_U) \twoheadrightarrow \cQ.
$$
Then we can form the diagram
\begin{equation} \label{Amindiag1}
\begin{gathered}
\xymatrix
{
& & \pi_{U}^{*} \O_C \ar^>>>>>{\pi_U^* s_\xi}[d] & & \\
0 \ar[r] & \cK \ar@{-->}[r] & \pi_U^* (\O_C)^{D}_{\O_U(S)} \ar@{-->}[r] \ar[d] & \cQ \ar[r] \ar@{=}[d] & 0 \\
& & \pi_U^*\ext^1(\F,\O_U) \ar[d] \ar@{->>}[r] & \cQ, & \\
& & 0 & &
}
\end{gathered}
\end{equation}
where $\cK$ is the kernel of the composition
$$
\pi_U^* (\O_C)^{D}_{\O_U(S)} \twoheadrightarrow \pi_U^*\ext^1(\F,\O_U) \twoheadrightarrow \cQ.
$$
We want to dualize the middle row of \eqref{Amindiag1}. Denote derived dual by $(\cdot)^\vee$. Since $C$ is Cohen-Macaulay we have \cite[Prop.~1.1.10]{HL}
$$
\ext^2(\O_C,\O_U)^{\vee} \cong \O_C[-2].
$$
Therefore applying $R\hom(\cdot,\pi_U^*\O_U(S))$ to the middle row of \eqref{Amindiag1} gives the following exact triangle (after a bit of rewriting)
$$
\cQ^\vee \otimes \pi_U^* \O_U(S) \rightarrow \O_{C \times B}[-2] \rightarrow \cK^\vee \otimes \pi_U^* \O_U(S),
$$
or in other words
\begin{equation} \label{trianglefam}
\O_{C \times B} \rightarrow \cK^\vee \otimes \pi_U^* \O_U(S)[2] \rightarrow \cQ^\vee \otimes \pi_U^* \O_U(S)[3].
\end{equation}
We claim that the induced map 
$$
\O_{U \times B} \rightarrow \cK^\vee \otimes \pi_U^* \O_U(S)[2]
$$
is a $B$-flat family of PT pairs. If so, then we have produced from a $B$-flat family of quotients a $B$-flat family of PT pairs and the inverse of the set theoretic map \eqref{setmap} is a morphism of schemes.
We claim that the complexes in \eqref{trianglefam} are all concentrated in degree 0. Indeed for any closed point $b \in B$ the pulled-back sheaves $\cK_b$ and $\cQ_b$ on the fibres $U \times \{b\}$ are pure sheaves of dimension 1 and 0. By \cite[Prop.~1.1.10]{HL}
\begin{align*}
\ext^{i \neq 2}_{\O_{U}}(\cK_b, \O_{U}) &= 0, \\
\ext^{i \neq 3}_{\O_{U}}(\cQ_b, \O_{U}) &= 0.
\end{align*}
Using \cite[Prop.~3.1]{Sch} we obtain
\begin{align*}
\ext^{i \neq 2}_{\O_{U \times B}}(\cK, \O_{U \times B}) &= 0, \\
\ext^{i \neq 3}_{\O_{U \times B}}(\cQ, \O_{U \times B}) &= 0.
\end{align*}
We obtain a short exact sequence
\begin{equation} \label{trianglefam2}
0 \rightarrow \O_{C \times B} \rightarrow \ext^{2}(\cK,\pi_{U}^{*} \O_U(S)) \rightarrow \ext^{3}(\cQ,\pi_{U}^{*} \O_U(S)) \rightarrow 0.
\end{equation}
Furthermore
\begin{align*}
&R\hom(\cK, \pi_U^* \O_U(S))[2], \ R\hom(\cQ,\pi_U^* \O_U(S))[3], \\
&R\hom(\cK_b,\O_U(S))[2], \ R\hom(\cQ_b,\O_U(S))[3],   
\end{align*}
are all concentrated in degree 0 for all closed points $b \in B$. Therefore the terms of \eqref{trianglefam2} are all $B$-flat by cohomology and base change for Ext groups \cite[Prop.~3.1]{Sch}. Finally the exact sequence \eqref{trianglefam2} pulls back to 
\begin{equation*}
0 \rightarrow \O_{C} \rightarrow \ext^{2}(\cK_b,\O_U(S)) \rightarrow \ext^{3}(\cQ_b,\O_U(S)) \rightarrow 0,
\end{equation*}
for all closed points $b \in B$. The second and third terms are pure of dimension 1 and 0 respectively, so we are done.
\end{proof}

\subsection{The higher rank case} \label{higherr}

In this section we prove Theorem \ref{key} for the case $r>2$. This requires a new inductive construction which builds on the rank $r=2$ case.

\begin{proof}[Proof of Theorem \ref{key} for $r>2$]
We set $M_r := M$ and start with the short exact sequence \eqref{cosection2}
$$
0 \rightarrow A(S)^{\oplus (r-1)} \rightarrow M_r(S) \rightarrow I_C \rightarrow 0.
$$
Let $A(S) \hookrightarrow A(S)^{\oplus (r-1)}$ be inclusion of the first factor and let $M_{r-1}(S)$ be the cokernel of the composition
$$
A(S) \rightarrow A(S)^{\oplus (r-1)} \rightarrow M_r(S).
$$
Then we obtain the following diagram in which all rows and columns are short exact sequences, all squares commute, and the bottom row is induced from the rest of the diagram ($3 \times 3$ lemma)
\begin{displaymath}
\xymatrix
{
& 0 \ar[d] & 0 \ar[d] & & \\
& A(S) \ar@{=}[r] \ar[d] & A(S) \ar[d] \\
0 \ar[r] & A(S)^{\oplus (r-1)} \ar[d] \ar[r] & M_r(S) \ar[d] \ar[r] & I_C \ar[r] \ar@{=}[d] & 0 \\
0 \ar@{-->}[r] & A(S)^{\oplus (r-2)} \ar[d] \ar@{-->}[r] & M_{r-1}(S) \ar@{-->}[r] \ar[d] & \ar@{-->}[r] I_C & 0.  \\
& 0 & 0 & &
}
\end{displaymath}
The factor $A(S)^{\oplus (r-2)}$ in the left column is the quotient of the inclusion $A(S) \hookrightarrow A(S)^{\oplus (r-1)}$ of the first factor.
This produces short exact sequences
\begin{align*}
0 &\rightarrow A(S) \rightarrow M_r(S) \rightarrow M_{r-1}(S) \rightarrow 0, \\
0 &\rightarrow A(S)^{\oplus (r-2)} \rightarrow M_{r-1}(S) \rightarrow I_C \rightarrow 0.
\end{align*}
Continuing inductively in this fashion, we obtain $A$-modules 
$$
I_C(-S) = M_1, \, M_2, \ldots, M_{r-1},  \, M_r = M,
$$ 
which fit in short exact sequences  
\begin{align} 
\begin{split} \label{crucial}
0 &\rightarrow A(S) \rightarrow M_{i+1}(S) \rightarrow M_i(S) \rightarrow 0, \\
0 &\rightarrow A(S)^{\oplus j} \rightarrow M_{j+1}(S) \rightarrow I_C \rightarrow 0,
\end{split}
\end{align}
for all $i = 1, \ldots, r-1$ and $j=0,\ldots, r-1$. From these short exact sequences we deduce at once that all modules $M_i$ are torsion free of homological dimension 1. We denote the corresponding extension classes by
$$
\xi_i \in \Ext^1(M_i,A).
$$
Dualizing \eqref{crucial} gives
$$
0 \rightarrow M_i^* \rightarrow M_{i+1}^* \rightarrow A \stackrel{s_{\xi_i}}{\rightarrow} \Ext^1(M_i,A) \rightarrow \Ext^1(M_{i+1},A) \rightarrow 0.
$$
Here the map $s_{\xi_i}$ sends 1 to $\xi_i$. 

Fix an element $Q \in \T_A$. The original cosection \eqref{cosection1} factors as
$$
\sigma : M(S) = M_r(S) \twoheadrightarrow M_{r-1}(S) \twoheadrightarrow \cdots \twoheadrightarrow M_1(S) = I_C.
$$
Using \eqref{crucial} and $\Ext^2(Q,A) = 0$ \cite[Prop.~1.1.6]{HL}, the inclusion of Lemma \ref{lem2} factors as a sequence of inclusions 
\begin{equation} \label{chainincl}
\Ext^2(Q,M_r(S)) \hookrightarrow \cdots \hookrightarrow \Ext^2(Q,M_1(S)) = \Ext^2(Q,I_C).
\end{equation}

\noindent \textbf{Step 1:}  Just like in Step 1 of Section \ref{rank2}, we ask when an element 
$$
I_C \rightarrow I^\mdot \rightarrow Q[-1]
$$ 
of $\Ext^1(Q[-1],I_C)$ lies in the image of $\Ext^2(Q,M_r(S))$ under inclusion \eqref{chainincl}. The same reasoning as in Step 1 of Section \ref{rank2} shows that this is the case if and only if there exist maps 
\begin{displaymath}
\xymatrix
{
\Ext^1(M_1,A) \ar@{->>}[r] \ar[d] & \Ext^1(M_2,A) \ar@{->>}[r] \ar@{-->}^{\exists}[dl] & \cdots \ar@{->>}[r] & \Ext^1(M_r,A) \ar@{-->}^{\exists}[dlll] \\
\Ext^3(Q,A(S)) & & &
}
\end{displaymath}
such that all triangles in the diagram commute. Here the first vertical map is induced by applying $\Hom(\cdot,A(S))$ to $I_C \rightarrow I^{\mdot} \rightarrow Q[-1]$ and recalling that $M_1 = I_C(-S)$. Note that if the maps exist, they are necessarily unique (because all the maps on the top row are surjections). In turn, this is equivalent to the existence of a single map 
\begin{equation}
\begin{gathered} \label{rankrdiag}
\xymatrix
{
\Ext^1(M_1,A) \ar@{->>}[r] \ar[d] & \Ext^1(M_2,A) \ar@{->>}[r] & \cdots \ar@{->>}[r] & \Ext^1(M_r,A) \ar@{-->}^{\exists}[dlll] \\
\Ext^3(Q,A(S)) & & &
}
\end{gathered}
\end{equation}
such that the triangle commutes. 

\bigskip

\noindent \textbf{Step 2:} We claim that an element $I_C \rightarrow I^{\mdot} \rightarrow Q[-1]$ of $\Ext^1(Q[-1],I_C)$ lies in the image of $\Ext^2(Q,M_r(S))^{\pure}$, i.e.~it lies in the image of $\Ext^2(Q,M_r(S))$ under inclusion \eqref{chainincl} \emph{and} corresponds to a PT pair, if and only if the map
$$
\Ext^1(M_r,A) \rightarrow \Ext^3(Q,A(S))
$$
in diagram \eqref{rankrdiag} is a surjection. This is proved just as in Step 2 of Section \ref{rank2}. 

\bigskip

\noindent \textbf{Step 3:} We have constructed a map 
\begin{equation} \label{higherrkmap}
\bigsqcup_{Q \in \T_A} \Ext^2(Q,M_r(S))^{\pure} / \sim \longrightarrow \Quot_A(\Ext^1(M_r,A)).
\end{equation}
Our goal in this step is prove that this is a geometric bijection. In fact we have constructed an entire sequence of maps
\begin{equation} \label{wholeseq}
\begin{gathered}
\xymatrix
{
\bigsqcup_{Q \in \T_A} \Ext^2(Q,M_r(S))^{\pure} / \sim \ar[r] \ar@{^(->}[d] & \Quot_A(\Ext^1(M_r,A)) \ar@{^(->}[d] \\
\bigsqcup_{Q \in \T_A} \Ext^2(Q,M_{r-1}(S))^{\pure} / \sim \ar[r] \ar@{^(->}[d] & \Quot_A(\Ext^1(M_{r-1},A)) \ar@{^(->}[d] \\
\cdots \ar@{^(->}[d] & \cdots \ar@{^(->}[d] \\
\bigsqcup_{Q \in \T_A} \Ext^2(Q,M_2(S))^{\pure} / \sim \ar[r] & \Quot_A(\Ext^1(M_2,A)).
}
\end{gathered}
\end{equation}
All squares commute. The vertical maps on the left are injective as observed in \eqref{chainincl}. Since the maps in the top row of diagram \eqref{rankrdiag} are all surjective, the vertical maps on the right are also injective. The bottom map is a geometric bijection, because we proved the rank 2 case in Section \ref{rank2}. Since \eqref{higherrkmap} is obtained as the restriction of the bottom map, we conclude that \eqref{higherrkmap} is a morphism of schemes. The diagram implies that all horizontal maps are injective. 

It is left to show that \eqref{higherrkmap} is surjective. Suppose we are given a surjection $\Ext^1(M_r,A) \twoheadrightarrow Q$, then we get induced maps
\begin{equation}
\begin{gathered} \label{induced} 
\xymatrix
{
\Ext^1(M_1,A) \ar@{->>}[r] \ar@{-->}^{\exists}[d] & \Ext^1(M_2,A) \ar@{->>}[r] \ar@{-->}^{\exists}[dl] & \cdots \ar@{->>}[r] & \Ext^1(M_r,A) \ar[dlll] \\
Q & & &
}
\end{gathered}
\end{equation}
Since the bottom map of diagram \eqref{wholeseq} is a geometric bijection, there exists a PT pair 
$$
I_C \rightarrow I^{\mdot} \rightarrow Q'[-1]
$$ 
in $\Ext^2(Q',M_2(S))^{\pure}$ mapping to $ \Ext^1(M_2,A) \twoheadrightarrow Q$ in diagram \eqref{induced}. In other words, there exists a commutative diagram  
\begin{equation*}
\begin{gathered}
\xymatrix
{
\Ext^1(I_C,A(S)) \ar@^{=}[r] \ar[d] & \Ext^1(M_1(S),A(S)) \ar[d] \\
\Ext^3(Q',A(S)) \ar^>>>>>>>>>>>{\cong}[r] & Q.
}
\end{gathered}
\end{equation*} 
where the left vertical map is induced by $I_C \rightarrow I^{\mdot} \rightarrow Q'[-1]$, the right vertical map appears in diagram \eqref{induced}, and we recall that $M_1=I_C(-S)$. Hence diagram \eqref{induced} reduces to diagram \eqref{rankrdiag}. This means $I_C \rightarrow I^{\mdot} \rightarrow Q'[-1]$ lies in $\Ext^2(Q',M_r(S))^{\pure}$ and it maps to $\Ext^1(M_r,A) \twoheadrightarrow Q$ because all vertical maps on the right of \eqref{wholeseq} are injective. Hence \eqref{higherrkmap} is a geometric bijection.
\end{proof}

\section{Hall algebra calculation} \label{Hall}

Let $X$ be a smooth projective threefold and let $\F$ be a rank $r$ torsion free sheaf on $X$ of homological dimension $\leq 1$. We assume $\Spec A$ is an affine scheme of one of the following types 
\begin{itemize}
\item \textbf{Case 1:} $\Spec A$ is a Zariski open subset of $X$,
\item \textbf{Case 2:} $\Spec A = \Spec \widehat{\O}_{X,P}$, where $P \in X$ is a closed point. 
\end{itemize}
We denote the $A$-module corresponding to $\F |_{\Spec A}$ by $M$
$$
M := \Gamma(\Spec A, \F|_{\Spec A}).
$$
Our goal is to prove equation \eqref{affinemaineq}, which finishes the proof of Theorem \ref{main} by Proposition \ref{redtolocal}. We will achieve this by combining the geometric bijection of Theorem \ref{key} with a higher rank variation on a Hall algebra computation by Stoppa-Thomas \cite{ST}, which we follow closely in this section. We assume $M$ has homological dimension 1, because otherwise $M$ is locally free in which case we already know \eqref{affinemaineq} (Section \ref{reductionsec}). Throughout this section, $X, \F, A, M$ are fixed in this way.

\subsection{Key lemma}

Like in Section \ref{section2}, we denote the stack of finitely generated 0-dimensional $A$-modules by $\T_A$. The following is the analog of \cite[Lem.~4.10]{ST}.
\begin{lemma} \label{lem1}
For any  $Q \in \T_A$, the only possibly non-zero Ext groups between $M$ and $Q$ are the following finite-dimensional vector spaces: $\Hom(M,Q) \cong \Ext^3(Q,M)^*$ and $\Ext^1(M,Q) \cong \Ext^2(Q, M)^*$. Moreover
$$
\dim \Hom(M,Q) - \dim \Ext^2(Q,M) = r \ell(Q).
$$ 
\end{lemma}
\begin{proof}
\textbf{Step 1:} Let $\iota : \Spec A \rightarrow X$ be  the natural map and $\cQ := \iota_* Q$. We claim 
\begin{align*}
\Ext^i_A(M,Q) &\cong \Ext^i_X(\F,\cQ), \\
\Ext^i_A(Q,M) &\cong \Ext^i_X(\cQ, \F),
\end{align*}
for all $i$. This will allow us to apply Serre duality on $X$ in Step 2. We prove the first isomorphism; the second follows analogously.

When $\iota : U=\Spec A \subset X$ is a Zariski affine open, the isomorphism follows at once from the local-to-global spectral sequence. Indeed
\begin{align*}
H^p(X,\ext^q_X(\F,\cQ)) \cong H^p(X, \iota_*\ext^q_U(\F|_U,Q)) \cong H^p(U, \ext^q_U(\F|_U,Q)),
\end{align*}
which is only non-zero when $p=0$ because $U$ is affine. The spectral sequence collapses thereby giving the desired result. 

In the case $\Spec A = \Spec \widehat{\O}_{X,P}$, we first observe that
\begin{align*}
\Ext^i_X(\F,\cQ)  \cong \Ext^i_{\O_{X,P}}(\F_{P},\cQ_{P})
\end{align*}
as complex vector spaces, where $\F_P, \cQ_{P}$ denote the stalks. This again follows by using the local-to-global spectral sequence on an affine open neighbourhood of the closed point $P$. Next let 
$$
j : \Spec \widehat{\O}_{X,P} \rightarrow \Spec \O_{X,P}
$$
be induced by formal completion. Then $\cQ_P = j_* Q$ by definition. Note that $j_*$ is exact because $j$ is affine and $j^*$ is exact by \cite[Prop.~10.14]{AM}. By adjunction 
\begin{align*}
\Ext^i_X(\F,\cQ)  \cong \Ext^i_{\O_{X,P}}(\F_{P},j_*Q) \cong \Ext^i_{A}(M,Q).
\end{align*}

\noindent \textbf{Step 2:} Since $\F$ has homological dimension $\leq 1$ and $\cQ$ is 0-dimensional, we have $\Ext^{\geq 2}_{X}(\F,\cQ) = 0$. By Serre duality on $X$ we also have $\Ext^{\leq 1}_{X}(\cQ,\F) = 0$ and the remaining Ext groups are $\Hom_X(\F,\cQ) \cong \Ext^3_X(\cQ,\F)^*$ and $\Ext^1_X(\F,\cQ) \cong \Ext^2_X(\cQ, \F)^*$. Furthermore
\begin{align*}
r \ell(\cQ) = \chi(\F,\cQ) &= \dim \Hom_X(\F,\cQ) - \dim \Ext^1_X(\F,\cQ) \\
&= \dim \Hom_X(\F,\cQ) - \dim \Ext^2_X(\cQ,\F).
\end{align*}
The result follows from Step 1.
\end{proof}

\subsection{The relevant stacks}

We write $\T_X$ for the stack of 0-dimensional sheaves on $X$. Similar to \cite{ST}, we consider the following $\T_X$-stacks
\begin{itemize}
\item $1_{\T_X}$ is the identity map $\T_X \rightarrow \T_X$,
\item $\Hom_X(\F, \cdot)$ is the stack whose fibre over $\cQ \in \T_X$ is $\Hom_X(\F, \cQ)$,
\item $\Hom_X(\F, \cdot)^{\onto}$ is the stack whose fibre over $\cQ \in \T_X$ is $\Hom_X(\F, \cQ)^{\onto}$,
\item $\Ext^2_X(\cdot,\F)$ is the stack whose fibre over $\cQ \in \T_X$ is $\Ext^2_X(\cQ, \F)$,
\item $\C^{r\ell(\cdot)}$ is the stack whose fibre over $\cQ \in \T_X$ is $\C^{r\ell(\cQ)}$.
\end{itemize}
Denote by
$$
\iota : \Spec A \rightarrow X
$$
the natural map. Recall that we introduced two cases at the beginning of this section. In Case 1, $\iota_*$ induces an isomorphism of $\T_A$ onto the \emph{open substack} of $\T_X$ of sheaves supported on $\Spec A \subset X$. In Case 2, $\iota_*$ induces an isomorphism of $\T_A$ onto the \emph{closed substack} of $\T_X$ of sheaves supported at the point $P \in X$. Recall that in the proof of Lemma \ref{lem1} we showed that
\begin{align*}
\Ext^i_A(M,Q) &\cong \Ext^i_X(\F,\cQ), \\
\Ext^i_A(Q,M) &\cong \Ext^i_X(\cQ, \F),
\end{align*}
for all $Q \in \T_A$, $i=0,1,2,3$, and $\cQ := \iota_* Q$. Pull-back along
$$
\iota_* : \T_A \hookrightarrow \T_X
$$
gives rise to the following $\T_A$-stacks:
\begin{itemize}
\item $\Hom_A(M, \cdot)$ is the stack whose fibre over $Q \in \T_A$ is $\Hom_A(M, Q)$,
\item $\Hom_A(M, \cdot)^{\onto}$ is the stack whose fibre over $Q \in \T_A$ is $\Hom_A(M, Q)^{\onto}$,
\item $\Ext^2_A(\cdot,M)$ is the stack whose fibre over $Q \in \T_A$ is $\Ext^2_A(Q, M)$.
\end{itemize}
In Case 1, these are open substacks of $\Hom_X(\F, \cdot)$,  $\Hom_X(\F, \cdot)^{\onto}$, $\Ext^2_X(\cdot,\F)$. In Case 2 these are closed substacks of $\Hom_X(\F, \cdot)$,  $\Hom_X(\F, \cdot)^{\onto}$, $\Ext^2_X(\cdot,\F)$. As such, we will view them as $\T_X$-stacks. 

Next we fix a cosection $M \rightarrow I_C(-S)$ as in \eqref{cosection2} (which always exist by Theorem \ref{Bourbaki} and Lemma \ref{lem0}!). By Lemma \ref{lem2} we get an induced injection
$$
\Ext^2_A(Q,M(S)) \hookrightarrow \Ext^2_A(Q,I_C) \cong  \Ext^1_A(Q,A/I_C),
$$
for all $Q \in \T_A$. In equation \eqref{locus} we defined $\Ext^2_A(Q, M(S))^{\pure}$ as the intersection of $\Ext^2_A(Q,M(S))$ with the locus of PT pairs $\Ext^1_A(Q,A/I_C)^{\pure}$. This gives a substack $\Ext^2_A(Q, M(S))^{\pure} \subset \Ext^2_A(Q, M(S))$. By applying $\cdot \otimes_A A(S)$, we obtain an isomorphism of stacks $\Ext^2_A(\cdot,M) \cong \Ext^2_A(\cdot,M(S))$ (where we use that $Q(S)$ is 0-dimensional for all $Q \in \T_A$). Pulling back the substack $\Ext^2_A(\cdot, M(S))^{\pure} \subset \Ext^2_A(\cdot,M(S))$ along this isomorphism gives a substack $$\Ext^2_A(\cdot, M)^{\pure} \subset \Ext^2_A(\cdot,M)$$ and clearly $\Ext^2_A(\cdot, M)^{\pure} \cong \Ext^2_A(\cdot,M(S))^{\pure}$. 
Furthermore, as above, we can view $\Ext^2_A(\cdot, M)^{\pure}$ as a substack of $\Ext^2_X(\cdot, \F)$. As such, $\Ext^2_A(\cdot,M)^{\pure}$ will also be viewed as a $\T_X$-stack. Although our constructions depend on the choice of cosection \eqref{cosection2} our final formula \eqref{affinemaineq} will not depend on this choice.

Denote by $H(\T_X) := K(\mathrm{St}/\T_X)$ the Grothendieck group of $\T_X$-stacks (locally of finite type over $\C$ and with affine geometric stabilizers). We refer to \cite{Bri1,Bri2,Joy1,Joy2,KS, ST} for general background on Hall algebra techniques. The group $H(\T_X)$ can be endowed with a product as follows. Let $\T^2_X$ be the stack of short exact sequences 
$$
0 \rightarrow \cQ_1 \rightarrow \cQ \rightarrow \cQ_2 \rightarrow 0
$$ 
in $\T_X$. For $i=1,2$, denote by $\pi_i$ the map that sends this short exact sequence to $\cQ_i$. We write $\pi$ for the map which sends this short exact sequence to $\cQ$. For any two ($\T_X$-isomorphism classes of) $\T_X$-stacks $[U \rightarrow \T_X]$ and $[V \rightarrow \T_X]$, the product $[U * V \rightarrow \T_X]$ is defined by the following Cartesian diagram
\begin{equation*}
\xymatrix
{
U*V \ar[d] \ar[r] & \T_X^2 \ar^{\pi_1 \times \pi_2}[d] \ar^{\pi}[r] & \T_X \\
U \times V \ar[r] & \T_X \times \T_X.
}
\end{equation*}
Then $(H(\T_X),*)$ is an associative algebra, known as a \emph{Joyce's motivic Ringel-Hall algebra}. The unit w.r.t.~$*$ is the $\T_X$-stack $1_0 := [\{0\} \subset \T_X]$ where $0$ denotes the zero sheaf. Let $\T_X^\circ = \T_X \setminus \{0\}$. Then 
$$
1_{\T_X} = 1_0 + 1_{\T_X^\circ}
$$ 
is invertible with inverse
$$
1_{\T_X} = 1_0 - 1_{\T_X^\circ} + 1_{\T_X^\circ} * 1_{\T_X^\circ} - \cdots
$$

We denote by 
$$
P_z(\cdot) : H(\T_X) \longrightarrow \Q(z)[\![q]\!]
$$
the virtual Poincar\'e polynomial. Here $z$ is the variable of the virtual Poincar\'e polynomial and $q$ keeps track of an additional grading as follows. Any element $[U \rightarrow \T_X] \in H(\T_X)$ is locally of finite type and can have infinitely many components. Let $\T_{X,n} \subset \T_X$ be the substack of 0-dimensional sheaves of length $n$ and define
$$
P_z(U) := \sum_{n=0}^{\infty} P_z(U \times_{\T_X} \T_{X,n}) \, q^n.
$$
Then $P_z(\cdot)$ is a \emph{Lie algebra homomorphism} to the abelian Lie algebra $\Q(z)[\![q]\!]$ by \cite[Thm.~4.32]{ST}.

\subsection{Proof of equation \eqref{affinemaineq}} \label{Hallfinal}

\begin{proof}[Proof of equation \eqref{affinemaineq}]
The inclusion $\iota_* : \T_A \hookrightarrow \T_X$ is a $\T_X$-stack, which we denote by $1_{\T_A}$. Just like $1_{\T_X}$, it is an invertible element of $(H(\T_X),*)$. 

By the inclusion-exclusion principle, $\Hom_A(M,Q)^{\onto}$ can be written as
$$
\Hom_A(M,Q) - \bigsqcup_{Q_1 < Q} \Hom_A(M,Q_1) + \bigsqcup_{Q_1 < Q_2 <Q} \Hom_A(M,Q_1) - \cdots,
$$
where $<$ denotes strict inclusion. This leads to Bridgeland's generalization of Reineke's formula in our setting (see \cite{ST, Bri2} for details)
\begin{equation} \label{BridgelandReineke}
\Hom_A(M,\cdot) = \Hom_A(M,\cdot)^{\onto} * 1_{\T_A},
\end{equation}
where we view all stacks as $\T_X$-stacks. We also use the following identity\footnote{The proof of this identity goes as follows. First we observe that $1_{\T_A} * \Ext^1_A(\cdot, A/I_C)^{\pure} = \Ext^1_A(\cdot,A/I_C)$ as in \cite{ST}. This comes from a geometric bijection from LHS to RHS. An object of LHS consists of a short exact sequence $0 \rightarrow Q_1 \rightarrow Q \rightarrow Q_2 \rightarrow 0$ in $\T_A$ together with a PT pair $0 \rightarrow A/I_C \rightarrow F \rightarrow Q_2 \rightarrow 0$. To these data, we assign the induced short exact sequence given by $0 \rightarrow A/I_C \rightarrow F \times_{Q_2} Q  \rightarrow Q \rightarrow 0$. This geometric bijection restricts to a geometric bijection $1_{\T_A} * \Ext^2_A(\cdot, M)^{\pure}  \rightarrow \Ext^2_A(\cdot,M)$. This follows by noting that $0 \rightarrow A/I_C \rightarrow F \rightarrow Q_2 \rightarrow 0$ satisfies the property described in Step 1 of Sections \ref{rank2} and \ref{higherr} if and only if $0 \rightarrow A/I_C \rightarrow F \times_{Q_2} Q  \rightarrow Q \rightarrow 0$ satisfies this property.} from \cite{ST}
$$
\Ext^2_A(\cdot,M) = 1_{\T_A} * \Ext^2_A(\cdot, M)^{\pure}.
$$
Using the fact that $P_z$ is a Lie algebra homomorphism we obtain
\begin{align*}
P_z \big(\Ext^2_A(\cdot,M)^{\pure} \big)(z^r q) &= P_z \big(1_{\T_A}^{-1} * \Ext^2_A(\cdot,M) \big)(z^r q) \\
&= P_z \big(\Ext^2_A(\cdot,M) * 1_{\T_A}^{-1} \big)(z^r q) \\
&=P_z \Big( \big(\Ext^2_A(\cdot,M) \oplus \C^{r\ell(\cdot)}\big) *  \big(\C^{r\ell(\cdot)} \big)^{-1} \Big)(q).
\end{align*}
Over strata in $\T_X$ where $\Hom_A(M,\cdot)$ is constant, the stacks 
$$
\Hom_A(M,\cdot), \ \Ext^2_A(\cdot,M) \oplus \C^{r\ell(\cdot)}
$$
are both Zariski locally trivial of the same rank by Lemma \ref{lem1}. We deduce
\begin{align}
\begin{split} \label{intermed1}
P_z \big(\Ext^2_A(\cdot,M)^{\pure} \big)(z^r q) &=P_z \Big( \big(\Ext^2_A(\cdot,M) \oplus \C^{r\ell(\cdot)}\big) *  \big(\C^{r\ell(\cdot)} \big)^{-1} \Big)(q) \\
&= P_z\big( \Hom_A(M,\cdot) * (\C^{r\ell(\cdot)})^{-1} \big)(q).
\end{split}
\end{align}
We also have
\begin{align}
\begin{split} \label{intermed2}
P_z\big(\Hom_A(M, \cdot)^{\onto} \big)(q) &= P_z \big(\Hom_A(M, \cdot) * 1_{\T_A}^{-1} \big)(q) \\
&= P_z \Big( \big(\Hom_A(M, \cdot) * (\C^{r\ell(\cdot)})^{-1} \big) * \big(\C^{r\ell(\cdot)} * 1_{\T_A}^{-1}\big) \Big)(q). 
\end{split}
\end{align}
Define 
\begin{align*}
U &:= \Hom_A(M, \cdot) * (\C^{r\ell(\cdot)})^{-1}, \\
V &:=\C^{r\ell(\cdot)} * 1_{\T_A}^{-1}.
\end{align*}
By \cite[Thm.~4.34]{ST}, if both $\lim_{z \rightarrow 1} P_z(U)$ and $\lim_{z \rightarrow 1} P_z(V)$ exist, then we have
\begin{equation} \label{lims}
\lim_{z \rightarrow 1} P_z(U*V) = \lim_{z \rightarrow 1} P_z(U) \, \lim_{z \rightarrow 1} P_z(V).
\end{equation}
By \eqref{intermed1} we have
\begin{align*}
\lim_{z \rightarrow 1} P_z(U)(q) &= \lim_{z \rightarrow 1} P_z \big(\Ext^2_A(\cdot,M)^{\pure} \big)(z^rq) \\
&= \sum_{n=0}^{\infty} e \big( \Quot_A(\Ext^1_A(M,A),n) \big) \, q^n, 
\end{align*}
where the second line follows from $\Ext^2_A(\cdot, M)^{\pure} \cong \Ext^2_A(\cdot,M(S))^{\pure}$ and Theorem \ref{key}. Since we have $\C^{r \ell(\cdot)} \cong \Hom_A(A^{\oplus r}, \cdot)$, the analog of \eqref{BridgelandReineke} with $M$ replaced by $A^{\oplus r}$ gives
\begin{align*}
\lim_{z \rightarrow 1} P_z (V)(q) &= \lim_{z \rightarrow 1} P_z \big(\Hom_A(A^{\oplus r}, \cdot)^{\onto} \big)(q) \\
&= \sfM(q)^{r e(\Spec A)}, 
\end{align*}
where the second line follows from Cheah's formula (see the proof of Proposition \ref{locfree}). Finally \eqref{intermed2} yields
\begin{align*}
\lim_{z \rightarrow 1} P_z(U*V) &= \lim_{z \rightarrow 1}P_z\big(\Hom(M, \cdot)^{\onto}\big)(q) \\
&= \sum_{n=0}^{\infty} e\big(\Quot_A(M,n)\big) \, q^n.
\end{align*}
Therefore equation \eqref{lims} implies the formula we want to prove
\begin{equation*} 
\sum_{n=0}^{\infty} e\big(\Quot_A(M,n)\big) \, q^n = \sfM(q)^{r e(\Spec A)} \sum_{n=0}^{\infty} e\big(\Quot_A(\Ext^1(M,A),n)\big) \, q^n. \qedhere
\end{equation*}

\vspace{-1.1cm}

\end{proof}

\vspace{0.5cm}

\bibliographymark{References}

\providecommand{\bysame}{\leavevmode\hbox to3em{\hrulefill}\thinspace}
\providecommand{\arXiv}[2][]{\href{https://arxiv.org/abs/#2}{arXiv:#1#2}}
\providecommand{\MR}{\relax\ifhmode\unskip\space\fi MR }
\providecommand{\MRhref}[2]{%
  \href{http://www.ams.org/mathscinet-getitem?mr=#1}{#2}
}
\providecommand{\href}[2]{#2}

\end{document}